\documentclass[a4paper,12pt]{article}

\usepackage{arxiv}

\usepackage[utf8]{inputenc}
\usepackage[T1]{fontenc}
\usepackage{hyperref}
\usepackage{url}
\usepackage{booktabs}
\usepackage{amsfonts}
\usepackage{nicefrac}
\usepackage{microtype}
\usepackage{graphicx}
\usepackage{natbib}
\usepackage{doi}
\usepackage{bm}
\usepackage{mathtools}
\usepackage{xcolor}
\usepackage{amsmath,amssymb,amsthm}
\usepackage{accents}

\usepackage{cleveref}

\newcommand{\reals}{\mathbb{R}}
\DeclareMathOperator{\Cov}{Cov}
\DeclareMathOperator{\Expectation}{\mathbb E}
\DeclareMathOperator{\Grad}{grad}
\DeclareMathOperator{\Maxexp}{\mathcal E}
\DeclareMathOperator{\Span}{Span}
\DeclareMathOperator{\Supp}{Supp}
\DeclareMathOperator{\Var}{Var}
\DeclareMathOperator{\Projective}{\mathbb P}
\newcommand{\KL}[2]{\operatorname{D}\left(#1\,\middle\Vert#2\right)}
\newcommand{\covat}[3]{\Cov_{#1}\left(#2,#3\right)}
\newcommand{\derivby}[1]{\frac{d}{d#1}}
\newcommand{\etransport}[2]{\prescript{\text{e}}{} {\mathbb U} _ {#1} ^ {#2}}
\newcommand{\expbundleat}[1]{S\maxexpat{#1}}
\newcommand{\expectat}[2]{\Expectation_{#1}\left[#2\right]}
\newcommand{\expfiberat}[2]{S_{#1}\maxexpat{#2}}
\newcommand{\expof}[1]{\exp\left(#1\right)}
\newcommand{\idealof}[1]{\operatorname{Ideal}(#1)}
\newcommand{\maxexpat}[1]{\Maxexp\left(#1\right)}
\newcommand{\mtransport}[2]{\prescript{\text{m}}{} {\mathbb U} _ {#1} ^ {#2}}
\newcommand{\normat}[2]{\left\Vert#2\right\Vert_{#1}}
\newcommand{\normof}[1]{\left\Vert#1\right\Vert}
\newcommand{\pderivby}[1]{\frac {\partial} {\partial #1}}
\newcommand{\rexpbundleat}[1]{S_{\text{r}} {\mathcal E}\left(#1\right)}
\newcommand{\scalarat}[3]{\left\langle#2,#3\right\rangle_{#1}}
\newcommand{\bracketat}[3]{\left\langle {#2} \middle| {#3} \right\rangle_{#1}}
\newcommand{\scalarof}[2]{\left\langle#1,#2\right\rangle}
\newcommand{\setof}[2]{\left\{#1 \, \middle| \, #2 \right\}}
\newcommand{\set}[1]{\left\{#1\right\}}
\newcommand{\simplexof}[1]{\Delta\left(#1\right)}
\newcommand{\spanof}[1]{\Span\left(#1\right)}
\newcommand{\supp}[1]{\Supp{#1}}
\newcommand{\uno}{\bm 1}
\newcommand{\varat}[2]{\Var_{#1}\left(#2\right)}
\newcommand{\velocity}[1]{\accentset{\star}{#1}}
\newcommand{\euler}{\mathrm{e}}

\usepackage{cleveref}

\newcommand{\urlNewWindow}[1]{\href[pdfnewwindow=true]{#1}{\nolinkurl{#1}}}

\theoremstyle{plain}
\newtheorem{thm}{Theorem}[section] 
\newtheorem{lem}[thm]{Lemma} 
\newtheorem{prop}[thm]{Proposition} 
 
\theoremstyle{definition} 
\newtheorem{defn}{Definition}[section] 
 
\newtheorem{ex}{Example}[section]

\title{The Fisher score on the closed simplex}

\date{}

\author{ 
    Giovanni Pistone \\
	Statistics Initiative\\
	Collegio Carlo Alberto\\
	Torino 10122, Italy \\
    \texttt{giovanni.pistone@carloalberto.org} \\
\And
    Fabio Rapallo \\
	Department of Economics\\
	University of Genova\\
	Genova 16126, Italy \\
    \texttt{fabio.rapallo@unige.it} \\
\AND
    Eva Riccomagno \\
	Department of Mathematics\\
	University of Genova\\
	Genova 16146, Italy \\
    \texttt{eva.riccomagno@unige.it} \\
}

\renewcommand{\headeright}{ }
\renewcommand{\undertitle}{ }

\begin{document}
\maketitle
\begingroup
\renewcommand{\thefootnote}{}
\endgroup

\begin{abstract}
  We extend classical analytic tools for finite-state statistical models to allow zero probabilities. Using methods from algebraic statistics and information geometry, we develop a framework in which a smooth statistical model can hit the boundary of the simplex, for example, in contingency tables with non-structural zeros. The central object of our approach is the vector bundle whose fibers are the $p$-contrasts associated to each probability distribution $p$.  In this framework, the Fisher score, and other key statistical concepts such as the entropy of one-dimensional statistical models, admit an algebraic representation also on the boundary of the simplex.
\end{abstract}


\section{Introduction}

This paper investigates statistical models for categorical data, with particular emphasis on one-parameter models defined on the probability simplex. We focus on geometric and algebraic properties of statistical models that are typically applied within the class of exponential family.  In this setting, a one-dimensional family of discrete probability distributions is seen as a smooth curve on the probability simplex. In turn, the probability simplex is seen as a closed convex subset of an affine subspace of the projective space. This geometric viewpoint provides a unified framework for the systematic and general study of the statistical, geometric, and algebraic properties of smooth statistical models.

This analysis relies on the relationship between the tangent space of the model curve and the contrast space, the latter defined as the vector space encoding the relevant log-linear variations in categorical models. Examples for contingency tables will be discussed throughout the paper. The connection between these two spaces is the key tool for characterising the models' local geometry and giving a statistical interpretation to the differential structures induced by {generalised} exponential parametrisations. Thus this paper sits within the interplay between algebraic statistics and information geometry, clarifying how concepts from both areas can be jointly applied to the study of statistical models. For other examples on such interplay see~\cite{gibilisco2010algebraic,heaton2025computing, MR3875134,
  MR3329739}.

Most existing literature addresses the open simplex, where each category has strictly positive probability. When structural or observed zeros arise, however, the closed simplex must be considered. This matters in compositional data analysis (e.g.~\cite{aitchison:1986}), where some `ingredients' may be missing from certain compositions yet should not be excluded from the analysis. The classical log-ratio approach to compositional data is ill-defined under structural zeros; extensions handling zeros are given e.g. in~\cite{martinfernandez2011zeros}. 

One way to handle structural zeros in a contingency table is to vectorise the table and exclude the cells corresponding to structural zeros. This is done, for example, in relational models~\cite{klimova2012relational,klimova2022testing} and in algebraic log-linear models~\cite{pistone|riccomagno|wynn:2001,drton|sturmfels|sullivan:2009}. Here we consider a slightly different approach by focusing on smooth one-parameter statistical models that, for some parameter values, lie on the boundary of the simplex. 

The main contribution of this paper consists in extending the geometric analysis of one-parameter exponential models from the open simplex to the closed simplex, thereby providing a coherent treatment that also includes boundary cases. To do this, the paper proposes an algebraic translation of objects that are traditionally viewed within the differential geometric approach, that is, in the open simplex. The generalisation of the Fisher score to the closed simplex is a key step, which is achieved through the exploitation of the relationships between the tangent space and the contrast space.

Information geometry as it was defined in the original monograph~\cite{amari|nagaoka:2000} is the study of the open probability simplex as an affine space. An affine space is classically defined via a notion of displacement as a mapping from a pair of points to a vector space as shown in~\cite{weyl:1952}, \cite[Ch.~2--3]{berger:1994}. On the open simplex, the displacement is the log-ratio $(p,q) \mapsto \log \frac q p - \expectat p {\log \frac q p}$ developed in~\cite{aitchison:1986} and~\cite{pistone|sempi:95}. The affine space allows for a proper calculus and the discussion of statistical models as differentiable sub-manifolds. The differential structure fits well with the algebraic statistics approach initiated by~\cite{diaconis|sturmfels:98} and \cite{pistone|riccomagno|wynn:2001}. In fact, monomial statistical models and their toric counterparts are represented in the affine coordinates as affine subspaces. Affine objects such as Segre varieties and Veronese varieties have a clear statistical interpretation. In turn, such an abstract construction provides exactly the mathematics needed by statistical physics and by Fisher statistical methodology up to the "Computer Age Statistical Methods" as described in~\cite{efron|hastie:2016}. This field currently has a substantial body of literature. In this paper we extend the scope of the affine space model from the open simplex to a larger affine space including the closed simplex. This is obtained by considering homogeneous coordinates of unnormalised distributions (compositions) $[f]$ as the special class in the projective space where $\sum f_x \neq 0$.  

The information geometry approach to statistical models adds to the affine structure a metric as a further essential structure. This specific applied area has developed a somewhat peculiar terminology after the seminal work of C. R. Rao \cite{rao:45} that might not be known to the typical reader of Algebraic Statistics. We thus mention here that information geometry treats a parametric family of probability distributions as a Riemannian manifold. Precisely, the covariance inner product acting on the tangent of the affine space produces the relevant metric. The parametric expression of such a metric is called  Fisher information matrix. For a full presentation, see \cite{ay|jost|le|schwachhofer:2017}; a very short tutorial is \cite{pistone:2020-NPCS}.

In fact, the tangent directions at each point of the parameter space are expressed in the affine space by the Fisher score, that is the gradient of the log-likelihood with respect to the parameters. The Fisher information matrix is then just the covariance of this score across the sample space, which is what makes it a natural — and, by Chentsov's theorem, essentially the unique — choice of Riemannian metric on a statistical manifold.

This turns a parametric statistical model into a curved geometric object.  The main consequences are that distances and divergences become geometric quantities, efficiency and the Cramér–Rao bound become curvature statements as shown in~\cite{rao:45}; exponential families are dually flat manifolds with a crucial role being played by the e-connection and m-connection, in some sense exponential families are ``the linear algebra of statistics'' (see for example~\cite{amari|nagaoka:2000}); the EM algorithm, iterative proportional fitting, and belief propagation algorithms can be understood as alternating projections (in the e- or m-sense) between submanifolds and steepest ascent in the Fisher metric is invariant to reparametrisation and often converges fast and has become standard in, e.g., neural networks and reinforcement learning (see e.g.~\cite{amari1995information}); a differential-geometric interpretation of AIC/BIC/MDL generalises to singular models where the usual asymptotic theory (based on local quadratic approximation) breaks down and singular learning theory in~\cite{watanabe:2009} addresses the case where the Fisher information matrix degenerates; recent work in~\cite{ay2007geometric} uses information geometry to characterise causal versus non-causal factorisations of a joint distribution.

The paper is organised as follows. Section~\ref{section:tgcontrast} addresses the relationship between tangent and contrast spaces, showing that the algebraisation of the process occurs through the contrast space and hinges on a condition on the supports. In particular, expressing the contrast space in coordinate form allows one to derive linear equations. {In Section~\ref{section-closed-probability-simplex} we recall the projective representation of the closed probability simplex.} 
In Section~\ref{sec:score} we show that Fisher scores are contrasts also on the boundary of the simplex. We use this result in Section~\ref{sec:statistical-bundle} to extend the statistical bundle on the closed simplex. Such an extension is exploited in Section~\ref{sec:natgradient}, and expanded in~\ref{sect:parallel}, to study the behaviour on the boundary of some classical objects of mathematical statistics and information geometry, such as the natural gradient and the parallel transport. In the concluding Section~\ref{sec:final}, we highlight the main findings of this paper and provide pointers to future work.
Appendix~\ref{appendices} collects examples, details, and results that aid understanding.

\section{Tangent bundle of the probability simplex and contrast spaces} \label{section:tgcontrast}

To fix notation, let $\reals$ be the set of real numbers, $\Omega$ a finite set  with $d$ elements and $p$ a vector of probabilities on $\Omega$, i.e. $p=(p_x  \, | \, x\in \Omega)$. For example, we could have a categorical random variable $X$ with $d$ categories and a random sample from $X$ that falls exactly into category $x\in \Omega$ with probability $p_x$, thus $p_x=\operatorname{Probability}(X=x)$.

\begin{defn} \label{def_contrasts}
Let $L(\Omega)=\reals^\Omega$ be the space of real-valued functions on $\Omega$, equivalently identified with $\reals^d$.
\begin{enumerate}
    \item The \emph{probability simplex} is the set of $d$-dimensional real-valued vectors whose components are non-negative and sum to one, 
\begin{equation*}
    \Delta(\Omega) = \setof {p \in \reals^\Omega}{p_x \geq 0, \sum_{x\in \Omega} p_x = 1}\ 
\end{equation*}
\item \label{item:contrasts} 
The space of \emph{contrasts} $C(\Omega)$ is the vector space of zero-sum vectors
\begin{equation*}
C(\Omega) = \setof {c \in \reals^{\Omega}}{ \sum_{x\in \Omega} c_x = 0} \ .
\end{equation*}
\item \label{item:p-contrasts} For $p\in \Delta(\Omega)$, the space of $p$-\emph{contrast} is the vector space 
\begin{equation*}
\setof {u \in \reals^{\Omega}}{\expectat{p}{u} = \sum_{x\in \Omega} u_x p_x = 0 } \ .
\end{equation*}
\end{enumerate}
\end{defn}

Items~2 and~3 of Definition~\ref{def_contrasts} provide two alternative versions of the tangent space of $\Delta(\Omega)$. 
The injection $\Delta(\Omega) \subset \reals^d$ provides the simplex with the structure of an affine space in the sense of \cite[\S~9]{bourbaki:70algebre} via the displacement function $(p,q) \mapsto q-p \in C(\Omega)$, this is because for $p,q\in \Delta(\Omega) $
\begin{equation*}
    \sum_{x\in \Omega} (q_x-p_x) = 0 \ .
\end{equation*}

The affine structure naturally gives rise to the notion of a vector space of velocities (see, for example, \cite[Ch.~3]{schwartz:1981}). 
In the following we consider a differentiable curve $t \mapsto \gamma(t)\in \Delta(\Omega)$ defined on an open interval of the real line including zero; thus for each $t$ in the open interval, $\gamma(t)$ is a $d$-dimensional vector with components $\gamma(x;t) $ with $x \in \Omega$, and it holds $\sum_{x\in\Omega} \gamma(x;t) = 1$ and  $\sum_{x\in\Omega} \dot \gamma(x;t) = 0$.
Curves are denoted $\gamma$, $\gamma(t)$ or $\gamma(x;t)$ depending on whether $x$ and $t$  are relevant to the context. 

Given a reference point $p \in \Delta(\Omega)$, the expression of the curve $\gamma$ with respect to the origin $p$ is $t \mapsto \gamma_p(t) = \gamma(t) - p $ which evolves in the vector space $C(\Omega)$. 
The curve $\gamma$ is seen as a one-parameter statistical model with parameter $t$ and $\gamma_p$ is a contrast with respect to the reference point $p$. 
The velocity of $\gamma$ {equals} the derivative of $\gamma_p$, which does not depend on $p$. 
This motivates the following definition of the space of velocities, which coincides with the space of contrasts as shown in Theorem~\ref{theorem_contrast_bundle} below. 

\begin{defn}
    \label{defn_tangentbundle}  
\ 
\begin{enumerate} 
\item The tangent space at $p \in \Delta(\Omega) $ is the set of velocities at $p \in \Delta(\Omega)$:
\begin{equation*}
 T_p\Delta(\Omega) =
  \setof{v \in \reals^{\Omega}}{\text{{there exists a differentiable }}
 \gamma \colon  t \mapsto \gamma(t) \in \Delta(\Omega) , \gamma(0) = p , \gamma \in C^1, \dot \gamma(0)=v} \ .
\end{equation*}  
\item The tangent spaces at each point $p$  are assembled into the tangent bundle, defined as
\begin{equation*}
 T\Delta(\Omega) = \setof{(p,v)}{ p \in \Delta(\Omega), v \in T_p\Delta(\Omega)} \ .
\end{equation*}
\end{enumerate}
\end{defn} 

Definition~\ref{defn_tangentbundle} is a very special case of a general structure in differential geometry (see, for example,~\cite{lang:1995}). Given the manifold $M$, the vector space $T_pM$ of velocities of all curves through $p \in M$ is called the tangent space at $p$, while the tangent bundle $TM$ is the bundle of all tangent spaces. In our case at each $p$ the vector space of velocities is called the fiber at $p$ and denoted $T_pM$.

Theorem~\ref{theorem_contrast_bundle} gives an equivalent definition of the tangent bundle {of the simplex}. Indeed, Equation~\eqref{eq:tangent-bundle} gives a definition of the tangent bundle without reference to the notion of velocity. It shows that the space of contrasts is not sufficient to define the velocities in $T \Delta(\Omega)$ and that it depends on whether the position $p$, to which the velocity is applied, is on the boundary. The key feature is the inclusion of the union of the contrast spaces associated with all open faces, namely, the condition on the supports, as exemplified in Subsection~\ref{coordinate_subspaces}.

\begin{thm} \label{theorem_contrast_bundle}
Let $\supp v$ be the support of a vector $v\in \reals^\Omega$, namely, the subset of $\Omega$ for which $v$ has a non-zero component. Then
    \begin{equation} \label{eq:tangent-bundle}
 T{\Delta}(\Omega) = \setof{(p,v)}{p \in {\Delta}(\Omega),
 v\in \mathbb R^{\Omega}, \sum_{x\in \Omega} v(x) = 0, \supp v {\subseteq} \supp p} \ ,
\end{equation}
\end{thm}
\begin{proof} We distinguish whether $p$ is on the boundary of the simplex or not. 
If $p$ is not on the boundary of the simplex (that is, $p_x>0$ for all $x\in \Omega$), then for any contrast  $v$ the linear curve  $\gamma(t) = p + tv$ belongs to the simplex for any $t$ in a neighbourhood of zero, $\gamma(0)=p$, and $\dot \gamma(0)=v$; thus the pair $(p,v)$ is in $T_p\Delta(\Omega)$.
(This defines the tangent bundle of the interior of the simplex as it is used in information geometry (see, e.g.~~\cite{pistone:2020-NPCS}).) A similar argument holds on the full simplex. If $p$ is on the boundary of the simplex and $v$ is a contrast such that $v(x)=0$ whenever $p_x = 0$, then $\gamma(t) = p + tv$ belongs to the simplex for any  $t$ in a neighbourhood of zero, $\gamma(0)=p$ and $\dot \gamma(0)=v$; thus also the pair $(p,v)$ is in $T_p\Delta(\Omega)$.
\end{proof}

If $p$ is a vertex of the simplex, the condition in Equation~\eqref{eq:tangent-bundle} is satisfied only by the vector $v$ with all components equal to zero, corresponding to a degenerate statistical model assigning probability one to the vertex $p$. The tangent space at the vertex is the zero vector space.
The definition of the tangent bundle given in Equation~\eqref{eq:tangent-bundle} is interesting because the contrast space is inherently algebraic in nature, in two distinct ways. First, the condition of summing to zero is algebraic, although not monomial. Second, each fiber is defined by the vanishing of specific coordinates, representing the simplest possible example of a monomial variety. These aspects are made precise in Section~\ref{coordinate_subspaces}. 

As a conclusion of this section, observe that when a one-parameter statistical model is a differentiable curve, then Theorem~\ref{theorem_contrast_bundle} allows us to translate the geometric notion of tangent and velocity into the algebraic notion of contrast space.

\subsection{Coordinate subspaces}\label{coordinate_subspaces}
The tangent bundle as defined in Equation~\eqref{eq:tangent-bundle} is not the traditional tangent bundle of differential geometry because it is a collection of spaces of different dimensions, one for each face of the simplex. Our definition aligns with an algebraic-geometry description, in that we add to the equations for the space of contrasts (a subspace of $\reals^{\Omega}$) the equations describing the defective supports (the monomial variety of coordinates). In this paper, we are particularly interested in models defined on the boundary of the simplex, and we show how this can be implemented with particular reference to contingency tables.

We proceed to further analyse the structure of Equation~\eqref{eq:tangent-bundle}. In fact, the definition of the tangent bundle of the full simplex requires the introduction of separate vector spaces for each face. 
The proof of Theorem~\ref{theorem_contrast_bundle} suggests that there is a natural decomposition of the tangent bundle along the open faces of the simplex, according to the zero set of elements in $\Delta(\Omega)$. 

Let $C(\Omega) =\setof{v\in \mathbb R^\Omega}{ \sum_{x\in\Omega} v(x)=0}$ be the space of contrasts. 
For each non-empty subset $I \subseteq \Omega $ define the space of contrasts with support in $I$ 
\begin{equation*}
 C_I(\Omega) = \setof{v \in C(\Omega)}{v(x) = 0 \text{ if } x \notin I} = \spanof{e_x-e_y, x,y \in I} \ ,
\end{equation*}
where $e_x$ is the $d$-dimensional vector with coordinate $1$ in position $x\in I$ and $0$ otherwise. 
Any $C_I(\Omega)$ defines a space of contrasts compatible with the characterisation of the tangent bundle in Theorem~\ref{theorem_contrast_bundle} for the specific set of structural zeros of $I$ in $p$.
Each $C_I(\Omega)$ is a linear sub-variety defined by the equations
\begin{equation*}
    \sum_{x\in \Omega} v(x) = 0  \quad \text{and} \quad v(x)=0 
    \text{ if } x \notin I \ ;
\end{equation*}
that is, it is the intersection of $C(\Omega)$ with the \emph{coordinate sub-variety} given by the zero conditions. We have an intersection of a linear variety with a monomial variety. In applied statistics, for example, in the theory of contingency tables, it is helpful to avoid implicit equations. By introducing a vector basis of the full contrast space, followed by algebraic elimination, one seeks a presentation that entirely avoids the contrast conditions as done in Example~\ref{ex:quadrato}. 

\begin{ex} \label{ex:quadrato}
Consider the 4-point case, $\Omega = \set{1,2}^2$. Then
\begin{equation*}
    C(\Omega) = \setof{u \in \reals^\Omega}{\sum_{x\in \Omega} u(x) = 0} = \uno^\perp
\end{equation*}
where $\uno$ is the $2\times 2$ matrix with $1$ in each entry and $\uno^\perp$ is its orthogonal space  with respect to the Frobenius scalar product of matrices. 
A vector basis of $C(\Omega)$ is
\begin{align*}
    c_{12} &= e_{11}-e_{12} = \begin{bmatrix}
        1 & -1 \\ 0 & 0
    \end{bmatrix}\\
    c_{21} &= e_{11}-e_{21} = \begin{bmatrix}
        1 & 0 \\ -1 & 0 \end{bmatrix}\\
    c_{22} &= e_{11}-e_{22} = \begin{bmatrix}
        1 & 0 \\ 0 & -1 \end{bmatrix}
\end{align*}
 
Thus a generic contrast can be written as
\begin{equation*}
    c = x\begin{bmatrix}
        1 & -1 \\ 0 & 0
    \end{bmatrix} + y \begin{bmatrix}
        1 & 0 \\ -1 & 0 \end{bmatrix} + z \begin{bmatrix}
        1 & 0 \\ 0 & -1 \end{bmatrix} = \begin{bmatrix}
            x+y+z & -x \\ -y & - z
        \end{bmatrix}
\end{equation*}
and the tangent planes to the four faces are defined by the equations
\begin{equation*}
    x+y+z = 0 \quad x = 0 \quad y = 0 \quad z = 0
\end{equation*}
respectively, which correspond to the polynomial ideal generated by
$xyz(x+y+z)$ in the polynomial ring $\mathbb R[x,y,z]$.
That is, the space of contrasts in the saturated model for a two-way contingency table
is a real-valued three-dimensional vector space, with basis given by $c_{12},c_{21},c_{22}$. When the coefficients satisfy the conditions $xyz(x+y+z)=0$, then the contrast belongs to a space which is tangent to one of the faces of the three-dimensional simplex.

The direct interpretation in terms of structural zeros of a contingency table can be exemplified as follows. 
Take, as an example, the index set $I=\{ (1,2) \} \subset \Omega $ given by $x=0$ and corresponding to  the set of contrasts $ C_{\{ (1,2) \}} = \left[ \begin{array}{cc} y+z & 0 \\ -y & -z  \end{array} \right]$. These are the velocities compatible with all contingency tables that have a structural zero in position $(1,2)$, and the statistical model
\begin{equation*}
\gamma(t) = t \ \left[ \begin{array}{cc} y+z & 0 \\ -y & -z  \end{array} \right]
+ \left[ \begin{array}{cc} p_{11} & 0 \\ p_{21} & p_{22} \end{array} \right] \in \Delta(\Omega) \ , 
\end{equation*}
with $p_{ij}$ strictly positive and $p_{11}+p_{21}+p_{22}=1$, 
is well defined for $t$ in a neighbourhood of zero.
\end{ex}

In this section we consider the vector space structure of the tangent bundle. Thus every face of the simplex has its own tangent bundle, with its own structural zeros relative to the full simplex. 
Curves on the simplex and statistical models that include points also on the boundary will be considered in the following sections, see in particular Example~\ref{ex:uniformmarg}.

\begin{ex}[Example~\ref{ex:quadrato} continued]
Another possible basis of the space of contrasts is given by
\begin{align*}
    c'_{12} &= e_{11}+e_{21}-e_{12}-e_{22} = \begin{bmatrix}
        1 & -1 \\ 1 & -1
    \end{bmatrix}\\
    c'_{21} &= e_{11}+e_{12}-e_{21}-e_{22} = \begin{bmatrix}
        1 & 1 \\ -1 & -1 \end{bmatrix}\\
    c'_{22} &= e_{11}+e_{22}-e_{12}-e_{21} = \begin{bmatrix}
        1 & -1 \\ -1 & 1 \end{bmatrix}
\end{align*}
where the last vector is the basic Markov move of the independence model, see for instance~\cite{diaconis|sturmfels:98}. {With the basis $\{c'_{12},c'_{21},c'_{22}\}$, the generic contrast can be expressed as
\begin{equation*}
    c = x\begin{bmatrix}
        1 & -1 \\ 1 & -1
    \end{bmatrix} + y \begin{bmatrix}
        1 & 1 \\ -1 & -1 \end{bmatrix} + z \begin{bmatrix}
        1 & -1 \\ -1 & 1 \end{bmatrix} = \begin{bmatrix}
            x+y+z & -x+y-z \\ x-y-z & -x-y+z
        \end{bmatrix}
\end{equation*}
and the tangent space corresponds to the polynomial ideal generated by
\begin{equation*}
(x+y+z)(-x+y-z)(x-y-z)(-x-y+z) \, .
\end{equation*}
}
\end{ex}

\begin{ex}
Consider an element $p$ in the face opposite to $(1,1)$ in the same set-up of Example~\ref{ex:quadrato}, that is, 
\begin{equation*}
   \setof{p \in \Delta(\Omega)}{p(1,1)=0}
\end{equation*}
By Theorem~\ref{theorem_contrast_bundle} the tangent space of a differentiable curve $\gamma: t\mapsto \gamma(t)$ with $\gamma(0)=p$  is in a one-to-one relationship with  the set of contrasts 
\begin{equation*}
    \setof{v \in C(\Omega)}{v(1,1)=0}
\end{equation*}
Similarly, for the other subset of vertices, giving an easy way to write the tangent space of one-parameter statistical models for contingency tables with a structural zero in $(1,1)$. 
\end{ex}

\section{The closed probability simplex as a subset of the  projective manifold} \label{section-closed-probability-simplex}

We present the probability simplex as the non-negative part of projective space (lines through the origin whose coordinates share a sign), which yields a formulation of the tangent space as the bundle of curve velocities at a point. Projective space is not an affine space, but it is a differentiable manifold with a standard atlas and, in particular,  velocities and  tangent spaces are defined via this atlas of charts.
Main references include
\cite[Ch.~4-5]{berger:1994}, \cite[Ch.~ii-iii]{lang:1995}, \cite[\S~2.4]{docarmo:1992}. 
Statistically, this presents probabilities as compositions  (compositional data or unnormalised weights as formalised in \cite{aitchison:1986}) and the projective formalism itself is well established in the literature. 

Let $\Projective$ be the projective space $\reals^\Omega_* / \sim$, where $\reals^\Omega_* = \reals^\Omega\setminus\set0$ and $\sim$ is the equivalence relation $f \sim \lambda f$ for a real number $\lambda \neq 0$. In homogeneous coordinates, the generic element of $\Projective$ is $f = [f_x \colon x \in \Omega] = \reals f$.
 In particular, $f$ is a composition if all the components of $f$ have the same sign.

\begin{ex}[Probability vs composition vs circle]
 \begin{figure}[t]
    \centering
    \includegraphics[width=.4\linewidth]{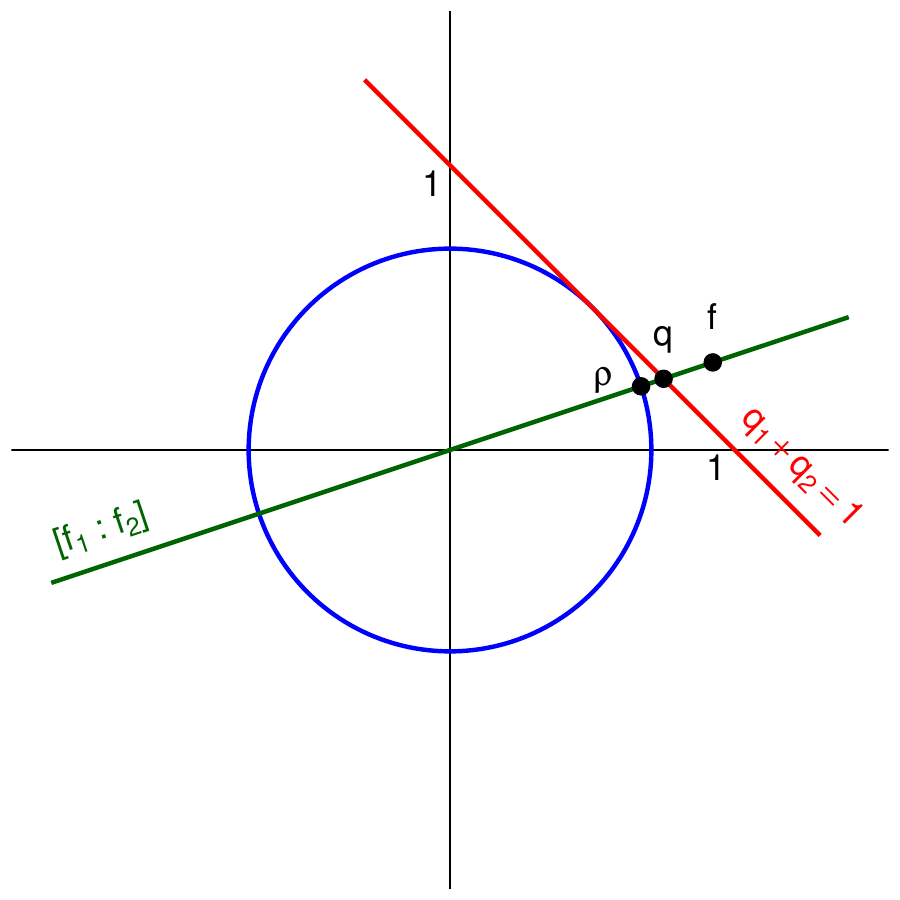}
    \caption{Homogeneous coordinates for $\Omega = \set{1,2}$. The line through the origin labelled $[f_1 \colon f_2]$ represents a point of $\Projective$. Its two natural representatives are $\rho$, on the sphere $\sum_y f_y^2 = 1/2$, and $q$, on the hyperplane $q_1+q_2=1$ tangent to the circle at $(1/2,1/2)$.
    \label{fig:projective}}
  \end{figure}
  Figure~\ref{fig:projective}  represents a generic probability density function $q$ in three ways: as a projective point $[f]$ with homogeneous coordinates $[f_1,f_2]$, as part of the affine subspace $\sum_y q_y = 1$, and as a point $\rho$ on the circle of radius $\sqrt{2}/2$. 
  A generic point of $[f]$, the probability $q$, and the point $\rho$ are collinear, and hence represent the same composition, related to one another by multiplication by a constant factor: 
$    q_y = \frac{f_y}{\sum_x f_x}$ and $ \rho_y = \frac{\sqrt 2} 2 \frac {f_y}{\normat {} {f}} $, where $\normat {} {\cdot}$ is the usual Euclidean norm.
The $q$-representation is defined only when $\sum_x f_x \neq 0$. This which excludes some projective points but suffices for all compositions.
\end{ex}

The classical results in Proposition~\ref{prop:projective-manifold} and Theorem~\ref{thm:tangent-bundle}—whose proofs are collected in the Supplementary material~\ref{appendix_proj} for completeness– endow the full simplex embedded in projective space with a proper geometric structure. The novelty of our presentation lies in including the boundary points of the simplex in this construction, and in showing that for polynomial models the relevant operations are algebraic.
Three examples illustrate the approach in a polynomial model, the full simplex and the R\'enyi entropy. 
This section concludes with the standard real Fubini–Study metric. 
Furthermore, Appendix~\ref{appendix_compositional} presents the application to compositional data.

\begin{prop}[Homogeneous atlas on $\Projective$] \label{prop:projective-manifold}
For $x \in \Omega$ let $V_x = \setof{[f] \in \Projective}{f_x \neq 0}$ and let $\phi_x \colon V_x \to \reals^{\Omega \setminus \set{x}}$ be the mapping
$  \phi_x[f_y \colon y \in \Omega] = (f_y/f_x , y \neq x) $. 
Then the domains $V_x$, $x \in \Omega$, together with the charts $\phi_x$ define an infinitely differentiable  
manifold on $\Projective$, namely  $(\Projective,V_x,\phi_x)_{x \in \Omega}$
is the tangent bundle of the projective manifold . 
Furthermore, all change-of-chart mappings are rational,
\begin{equation*}
  \phi_y \circ \phi_x^{-1} (a_z , z \neq x) = (a_z/a_y , z \notin \set{x,y}, a_x=1/a_y) \ .
\end{equation*}
\end{prop}

The tangent bundle of a projective manifold admits a synthetic description, meaning a description that avoids charts and uses only the equivalence map \(f \mapsto [f]\). A general theory is presented in \cite[\S~3]{milnor|stasheff:1974}, of which ours is a special case.
The version used in Information Geometry identifies the probability simplex with the non-negative part of the unit sphere via the square-root representation: an element $p$ of the simplex maps to a point $\rho \propto p$ on the sphere, and the tangent bundle of the sphere is identified with that of the simplex, as in \cite{amari|nagaoka:2000} and \cite{ay|jost|le|schwachhofer:2017}. Our approach differs somewhat, as outlined below.

There are two bundles inside {the trivial bundle} $\Projective \times \reals^\Omega$: the {tautological line bundle}, say $\gamma$, with fiber over the line $\reals f = [f]$ equal to the line itself, and the orthogonal complement bundle, say $\gamma^\perp$, of the pairs $([f], (\reals f)^\perp)$. 
For $u,v\in \ \reals^\Omega$, $f\in \Projective$ and $\lambda \in \reals$, orthogonality is defined by a given inner product $ \bracketat . . . $ compatible with the projective structure {indicated with a vertical bar  and} so that
  \begin{equation} \label{homoprod}
    \bracketat {\lambda f} {\lambda u} {\lambda v} = \bracketat f u v \ .
  \end{equation}
 {The corresponding norm is indicated as $\bracketat f \cdot \cdot$.}
 
The projections of $([f],v) \in \Projective \times \reals^\Omega$ are $([f],\lambda f)$ on the tautological bundle and $([f],\hat v)$ on the orthogonal complement bundle, with
\begin{equation} 
\bracketat f {v - \lambda f}{f} = 0 \Rightarrow \lambda = \frac{\bracketat f v f}{\bracketat f f f} 
\  \quad \text{and} \quad \  
\hat v = v - \lambda f = v -  \frac{\bracketat f v f}{\bracketat f f f} f \label{eq:decomposition} 
\end{equation} 

\begin{ex}
For $f\in \reals^\Omega$ such that $\sum_{x \in \Omega} f_x\neq 0$ and $u,v\in \reals^\Omega$, the expression 
\begin{equation*}
\frac{\sum_x u_x v_x f_x}{(\sum_x f_x)^3}
\end{equation*}
defines the weighted standard inner product compatible with the projective structure. Notice that the actual form of the normalizing factor is irrelevant, provided it is homogeneous. The equation 
\begin{equation*}
  \frac{\sum_{\setof x {f_x \neq 0}} u_x v_x f_x^{-1}}{(\sum_x f_x)}
\end{equation*}
defines a different inner product, the Fisher inner product.
\end{ex}

Theorem~\ref{thm:tangent-bundle} gives a characterisation of the tangent bundle of the projective manifold obtained by considering the quotient mapping in a chart and studying its kernel. 

\begin{thm} \label{thm:tangent-bundle}
The tangent bundle of the projective manifold $(\Projective,V_x,\phi_x)_{x \in \Omega}$
is isomorphic to $L(\gamma,\gamma^\perp)$; that is, a tangent vector at $[f]$ is a linear map $S$ from the line $\reals f$ into the plane $(\reals f)^\perp$. The homogeneous chart $(V_x,\Phi_x)$ lifts to the chart
  \begin{equation*}
  L([f],(\reals f)^\perp) \ni   S \mapsto \left(\frac{S_y(f)f_x - f_y S_x(f)}{f_x^2}\right)_{y \neq x} \in \reals^{\Omega \setminus \set{x}}\ .
  \end{equation*}
\end{thm}

\begin{ex}[Derivatives] \label{ex:derivatives}
  \begin{figure}[t]
    \centering
    \begin{tabular}{cc}
    \includegraphics[width=.4\linewidth]{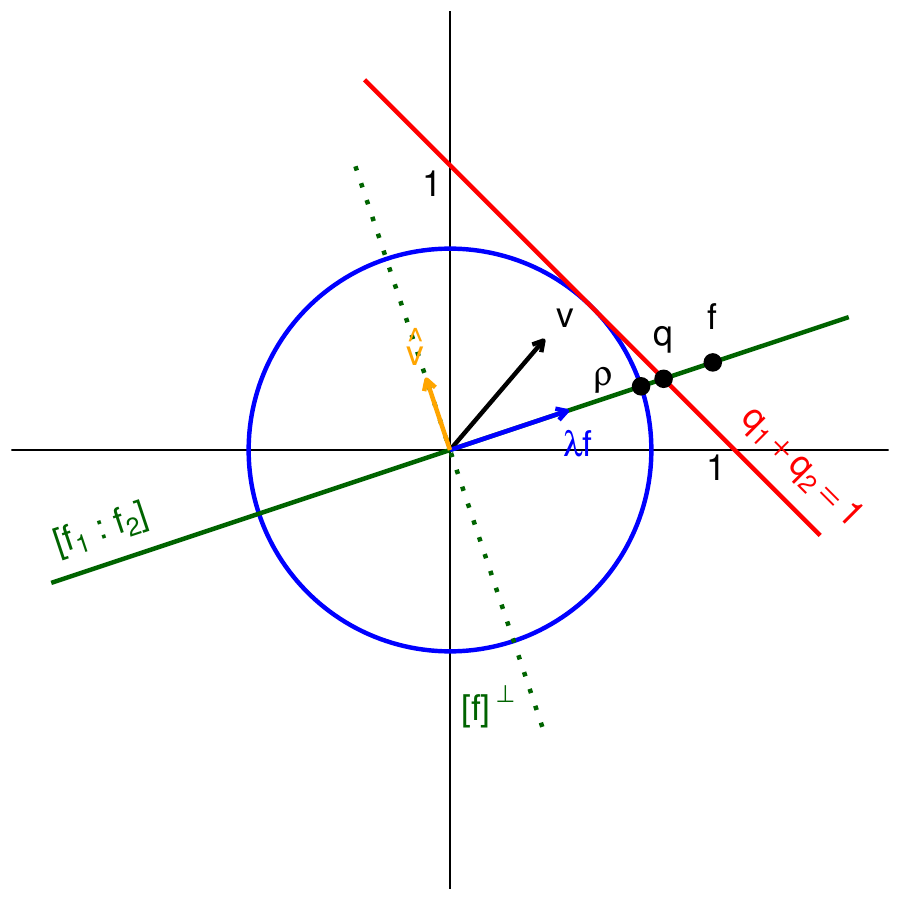}    & \includegraphics[width=.4\linewidth]{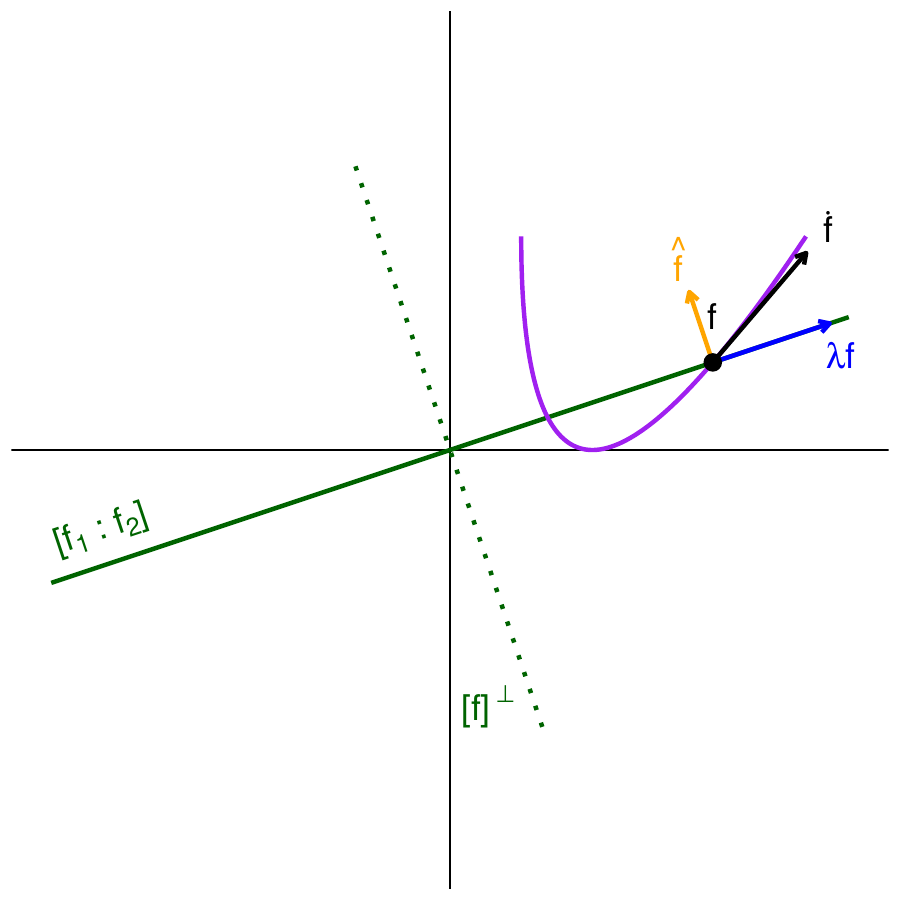}
    \end{tabular}
    \caption{Left: The line through the origin with label $[f_1 \colon f_2]$ represents the projective point $[f] \in \Projective$ of Figure~\ref{fig:projective}. The black dot $f$ is a representative of $[f]$, and the dotted line with label $[f]^\perp$, through the origin and orthogonal to $[f_1 \colon f_2]$, is the fiber $(\reals f)^\perp$ of the tangent bundle at point $[f]$.  The black vector $v$, applied at the origin, is orthogonally decomposed by $v = \lambda f + \hat v$. Right: The curved line is the support of the one-dimensional model of Example~\ref{ex:derivatives}. Equation~\eqref{eq:decomposition} is shown alongside it, with $\lambda f$ in blue, along the line $\reals f$, and $\hat v$ in orange, along $[f]^\perp$. The vector $\dot f$ is presented as applied at $f$ together with the decomposed factors $\lambda f$ and $\hat f$.
    \label{fig:projective-derivative}}
  \end{figure}
As shown in the left panel of Figure~\ref{fig:projective-derivative}, a representative $f$ of $[f]$ decomposes a vector $v \in \reals^{\set{1,2}}$ into a component $\lambda f$ along the line $\reals f$ and a component $\hat v$ along the orthogonal line $[f]^\perp$. The vector $\hat v$ is the orthogonal projection of $v$ onto $(\reals f)^\perp$, pictured as the dotted line through the origin.
  The right panel of Figure~\ref{fig:projective-derivative} shows the support of the one-dimensional polynomial model defined by 
\[  f_1(t) = t + (t-1/2)^2 = t^2 + \frac14 \ \quad \text{ and } \quad \ f_2(t) = 3(t-1/2)^2 \ . \]
which take values $(1/2,0)$ at $t = 1/2$ and is strictly positive otherwise.
  To present it as a probability, we normalise by $f_1 + f_2 = 4t^2-3t+1 > 0$, obtaining a rational representation. 
 
  With the standard inner product, all computations associated to the orthogonal decomposition of the velocity $\dot f$ are rational and can be carried out using computer algebra software. Here we use the Python library \texttt{sympy}.  The derivative is $ \dot f(t) = \bigl(2t, \, 3(2t-1)\bigr) $ and takes value $(1,0)$ at $t = 1/2$ because the curve is tangent to the horizontal axis.

The derivative itself factors component-wise as $\dot f = \velocity f \  f$, namely the Fisher score of Section~\ref{sec:score}, with
$    \velocity f_1(t) = \displaystyle \frac{8t}{4t^2+1}$ and  $\velocity f_2(t) = \displaystyle \frac{4}{2t-1}$.
Notice that $\velocity f_2$ has a pole at $t=1/2$, since $f_2(1/2)=0$.
Next we compute its orthogonal projection using~\eqref{eq:decomposition}. The projection onto the line $[f]$ is $ \lambda f$ where
\begin{equation*}
    \lambda(t) = \frac{\scalarof {\dot f} f}{\scalarof f f} = \frac{2\left(80t^3-108t^2+56t-9\right)}{80t^4-144t^3+112t^2-36t+5} \ .
\end{equation*}
and the projection $\hat f = \dot f - \lambda f$ onto $[f]^\perp$ is
\begin{equation*}
    \hat f(t) = \frac{1}{2\left(80t^4-144t^3+112t^2-36t+5\right)} \Bigl(-9(2t-1)^3(2t+1), \ 3(2t-1)(2t+1)(4t^2+1)\Bigr) \ ,
\end{equation*}
which indeed satisfies $\scalarof {\hat f} f = 0$.
Equivalently, $\hat f = (s_1 f_1, \, s_2 f_2)$ componentwise, with
\begin{equation*}
    s_1(t) = \frac{-18(2t-1)^3(2t+1)}{(4t^2+1)\left(80t^4-144t^3+112t^2-36t+5\right)} 
    \quad \text{and} \quad
    s_2(t) = \frac{2(2t+1)(4t^2+1)}{(2t-1)\left(80t^4-144t^3+112t^2-36t+5\right)} \ .
\end{equation*}
Notice that $s_2$ has a pole at $t=1/2$: since $f_2(t) = 3(t-1/2)^2$ vanishes there to second order while $\hat f_2(t)$ vanishes only to first order, $s_2(t) \to \infty$ as $t \to 1/2$, so this componentwise representation is only valid for $t \neq 1/2$.
\end{ex}

\begin{ex}[R\'enyi entropy]
Assume the representative of the homogeneous coordinate in the curve $t \mapsto [p(t)]$ is a probability function. Hence, $(\reals p(t))^\perp = \setof{w { \in \Delta(\Omega) } }{\expectat {p(t)} {w} = 0}$ and the projection of $([p],v)$ is $([p],\hat v)$,
  \begin{equation*}
    \hat v = v - \frac {\expectat p v} {\expectat p p} p \ .
  \end{equation*}
  In particular, {applying~\eqref{eq:decomposition} to $v = \dot p$ returns}
\begin{equation*}
\widehat {\dot p(t)} = \dot p(t) - \frac {\expectat {p(t)} {\dot p(t)}}{ \expectat {p(t)} {p(t)}} {p(t)} 
= \dot p(t) - \left(\frac 12 \derivby t \log \sum_y {p(t)}_y^2\right)  {p(t)} 
=
\derivby t \left(p(t) + \frac12 H_2(p(t) )\right)
\end{equation*}
where  $H_2(p) = - \log \normof {p}^2$ is the R\'enyi entropy of order 2, as defined, for example, in \cite{renyi:1970}. 
When the curve is not normalised, the Renyi entropy of order 2 on $\Projective$ is
  \begin{equation*}
    H_2[f] = - \log \sum_y f_y^2 + \log \left(\sum_x f_x\right)^2 \ .
  \end{equation*}
Using the decomposition \eqref{eq:decomposition} we find the derivative of Renyi entropy along a curve $t \mapsto [f(t)]$,
  \begin{equation*}
    \derivby t H_2[f(t)] = 2 \frac {\sum_x \widehat{\dot f(t)}}{\sum_x f(t)} \ .
  \end{equation*}
  For a generic tangent vector $v \in L([f],[f]^\perp)$, the value of the derivative is
  \begin{equation*}
      dH_2[f] \colon v \mapsto 2 \frac1{\sum_x f_x} v(f) \ .
  \end{equation*}
\end{ex}

So far, we have shown how to compute the derivative, and hence define the tangent space, in full generality, {in the simplex as subset of the projective space}.  The derivative is always defined, according to Theorem~\ref{thm:tangent-bundle}. In fact, the mapping $t \mapsto \normof f(t) ^2$ is never zero and is differentiable.

The representation of tangent vectors at $[f]$ as elements of $L([f],[f]^\perp)$ allows a construction permitting composition that may vanish at some points. Precisely, if the mapping $t \mapsto \normof{f(t)}$ is differentiable, one can always write $f(t) = \normof{f(t)} \rho(t) = 2 \normof {f} \sqrt p$ and use linearity, so that 
\begin{equation*}
    \widehat f = \widehat{\normof {f} \rho} = \normof {f} \widehat \rho = \normof f \dot \rho 
\end{equation*}
{using the symbols in Figures~\ref{fig:projective} and~\ref{fig:projective-derivative}.}

The trivial bundle $(\reals^\Omega \setminus \set{0} ) \times \reals^\Omega$ carries the standard inner product $\scalarof u v = \sum_x u_x v_x$ and also other inner products of interest. Proposition~\ref{propRiemannianmetric} below provides the Fubini--Study metric on $\Projective$ built out of the standard inner product. We use the realisation $T_{[f]}\Projective = L([f],f^\perp)$.

\begin{prop} \label{propRiemannianmetric}
For each $[f] \in \Projective$ and $u,v \in L([f],f^\perp)$, a Riemannian metric on $T_{[f]}\Projective$ is defined by 
\begin{equation*}
 \bracketat {[f]} {u} {v}  
 = \frac{\scalarof {u(f)} {v(f)}}{\scalarof {f} {f}}   \ .
\end{equation*}
\end{prop}
\begin{proof} This is independent of the choice of $f$. Note that ${u(f)}$ and ${v(f)}$ are orthogonal to $f$ and thus $f$-centred. Replacing $f$ by $\mu f$ scales
  $u(\mu f) \mapsto \mu u(f)$, $v(\mu f) \mapsto \mu v(f)$ because of the linearity. \end{proof}

The following statement provides the Riemannian metric of information geometry together with the precise connection between the standard notion of Fubini--Study metric and its version as known in information geometry. For the classical complex case we refer to~\cite[Ch.~0.2]{griffiths|harris:1978} and to \cite{kass:89} for the statistical version.

Proposition~\ref{prop:Fisher metric} below gives the decomposition in Equation~\eqref{eq:decomposition} with respect to the Fisher inner product.

\begin{prop}[Fisher metric]\label{prop:Fisher metric}
  For each $[f] \in \Projective$ such that $\sum_x f_x \neq 0$ {and $u,v \in L([f],f^\perp)$}, the inner product
  \begin{equation*}
    \bracketat f u v = (\sum_x f_x)^{-1} \sum_{\setof x {f_x \neq 0}} u_x(f) v_x(f) f_x^{-1}
  \end{equation*}
  is compatible with the projective structure. With respect to this inner product, the decomposition of the velocity is
  \begin{equation*}
    \widehat f = \dot f - (\sum_x f_x)^{-1}
    \left(\sum_{\setof x {f_x \neq 0}} \dot f_x\right)  f \ . 
  \end{equation*}
  If, moreover, $f_x \geq 0$, then
  \begin{equation*}
    \widehat f = \dot f - \derivby {t} {\left(\log \sum_x f_x \right)} f \ .
  \end{equation*}
\end{prop}

\begin{proof}
  We need only to discuss the last case. If $f_x \geq 0$, then $f_x = 0$ implies $\dot f_x = 0$ because the point is a minimum. Hence
  \begin{equation*}
    (\sum_x f_x)^{-1} \left(\sum_{\setof x {f_x \neq 0}} \dot f_x\right) = (\sum_x f_x)^{-1} \left(\sum_x \dot f_x\right) = \derivby {t} {\left( \log \sum_x f_x\right)}
  \end{equation*}
\end{proof}

    \section{Fisher score} \label{sec:score}

The score is a crucial component of Fisherian statistics~\cite{fisher:1922}. We start this section by arguing that the score can be interpreted as a velocity in a generalised affine space (see~\cite{pistone:2020-NPCS}), by applying the following definition to a one-dimensional differentiable statistical model in the open simplex, indicated with $\overset{\circ}{\Delta}(\Omega)$. Next, we generalise it to the closed simplex in Theorem~\ref{prop:score}. Affine spaces are systematically used in mathematical modelling, starting from~\cite{weil:1946} (initially published in 1918). 

\begin{defn}\label{def:displacement}
\ 
\begin{enumerate}
\item The displacement from $p$ to $q$ both in $\overset{\circ}{\Delta}(\Omega)$ is defined as
\begin{equation*}
    s_p(q) = \log \frac q p - \expectat p {\log \frac q p} {\in \mathbb R^\Omega}
\end{equation*}
where $\overset{\circ}{\Delta}(\Omega)$ is the interior of the simplex and the ratios are defined componentwise.

\item The exponential transport of a random variable $u\in \mathbb R^\Omega$ such that $\expectat p u = 0$ from $p\in \overset{\circ}{\Delta}(\Omega)$ in $q\in \overset{\circ}{\Delta}(\Omega)$ is defined as 
\begin{equation*}
    \etransport p q u = u - \expectat q u \, .
\end{equation*}    
\end{enumerate}
\end{defn}
In Appendix~\ref{sect:parallel} the notion of transport is discussed in greater detail.
The exponential transport at $q$ of the random variable $u$ is centred at zero with respect to $q$. In particular, the transport from $p$ to $q$ maps the $p$-contrasts to the $q$-contrasts, where a $p$-contrast is any random variable $u$ with $\expectat p u =0$ (see Definition~\ref{def_statEscorebundles}).

Next, for $p,q,r\in \overset{\circ}{\Delta}(\Omega)$ the ``generalised parallelogram rule'' holds
\begin{equation*}
s_p(q) + \etransport q p s_q(r) = s_p(r) \, .
\end{equation*}

Let $t \mapsto \gamma(t)$ be a differentiable curve in the open simplex. The expression of the curve $\gamma$ with origin at $p$ (that is, $\gamma(0)=p$) is
\begin{equation*}
    t \mapsto \gamma_p(t) = s_p(\gamma(t)) = \log \frac {\gamma(t)} p - \expectat p {\log \frac {\gamma(t)} p}
\end{equation*}
and the expression of its first derivative is
\begin{equation*}
    \derivby t \gamma_p(t) = \derivby t \log \gamma(t) - \expectat p {\derivby t \log \gamma(t)} \, .
\end{equation*}
Note the analogy with Section~2, which is also reflected in the notation.
In the usual affine setup the displacement from $p$ to $q$ is $q-p$, hence the derivative $\derivby t (\gamma(t) - p) = \dot \gamma(t)$ does not depend on the origin $p$, as mentioned in Section~2. 
We resort to the moving frame derivative, which does not depend on the origin,
\begin{equation*}
\left. \derivby t \gamma_p(t) \right|_{p=\gamma(t)} = \derivby t \log \gamma(t) \ ,
\end{equation*}
because $\left. \expectat p {\frac {\dot \gamma(t)}p}\right|_{p=\gamma(t)} = \expectat {\gamma(t)}{\derivby t \log \gamma(t)} = 0$.

\begin{defn}
For a one-parameter statistical model given by a differentiable curve   $t \mapsto \gamma(t)$  in the open simplex, the score function is the derivative of $\log \gamma(t)$, namely $\derivby t \log \gamma(t) $.
\end{defn}

Thus, in the precise sense above, the score is a velocity in an affine space where the velocity is the relative rate of change. Theorem~\ref{prop:score} shows how to extend this approach to the closed simplex.

\begin{thm} \label{prop:score}
If $t \mapsto \gamma(t) \in \reals^\Omega$ is a differentiable curve in $\Delta(\Omega)$, then its (ordinary) velocity is a curve $t \mapsto \dot \gamma(t)$ in the space of contrasts $C(\Omega)$. Moreover, for each $t$ there exists $s(t) \in \reals^\Omega$ such that 
\begin{equation}\label{eq:scoreA}
    \dot \gamma(t) = s(t) \ \gamma(t) \quad \text{and} \quad \supp {\dot \gamma(t) }= \supp s(t) \cap \supp \gamma(t)
\ ,\end{equation}
where the first equality in Equation~\eqref{eq:scoreA} is component-wise, i.e. $\dot \gamma(x;t) = s(x;t) \ \gamma(x;t)$  for all $x\in \Omega$.
\end{thm}
\begin{proof}If for some $x\in \Omega$ and some $\bar t$ it holds $\gamma(x;\bar t) = 0$, that is, the curve hits a face of the simplex at time $\bar t$, then the real-valued function $t \mapsto \gamma(x;t)$ has a minimum at $t=\bar t$, so that $\dot\gamma(x;\bar t) = 0$. This proves the absolute continuity of $\dot \gamma $ with respect to $\gamma$, in short 
$\dot \gamma \ll \gamma$, for each $x\in \Omega$. The conclusion follows from an elementary application of the Radon-Nikodym theorem.
\end{proof}

The statistical argument above is a particular case of a general convexity argument. Namely, if $t \mapsto x(t)$ is a differentiable curve in a closed convex set $K \subset \reals^\Omega$ and $x(\bar t) \in \partial K$, then $\dot x(\bar t)$ belongs to one of the supporting hyperplanes at $x(\bar t)$. Thus, the notion of score fits into the theory of sub-differentials as systematically developed after~\cite{rockafellar:1970}. However, we aim to find an algebraic setup that leverages the model's smoothness. 
The ordinary differential equations in Equations~(\ref{eq:scoreA}) are akin to the ODE obtained by differentiating the exponential function $f(x) = \exp(s \ x)$ with respect to $x$.
For $a,b$ and $\gamma$ such that the computations make sense, we have 
$\frac{d}{dt} \gamma (t) = s(t) \ \gamma(t)$ with solution $\gamma(b)  = \gamma(a)  \exp \left( \int_a^b  s(t) \ dt 
\right)  $. In this case $\gamma$ at $a$ and $b$ has the same support, that is, the ODE approach assumes constant support, while this paper is concerned with varying support.

If the curve $\gamma$ stays in the interior of the probability simplex, then the original Fisher's definition of the score $s(t)$ holds,
\begin{equation} \label{eq:score}
    s(t) = \frac{\dot \gamma(t)}{\gamma(t)} = \derivby t \log \gamma(t) \, .
\end{equation}
Otherwise, $s(t)$ is uniquely defined on $\supp \gamma(t)$ and arbitrary on the zero-set of $\gamma(t)$.

Let us discuss the notion of score from the point of view of the theory of singularities of statistical models as in~\cite{watanabe:2009}. If the model $\gamma$ is differentiable, then at each point $x \in \Omega$ it holds
\begin{equation*}
  \gamma(x;t) -\gamma(x;s) = (t-s) R_1(x;s,t) \quad \text{where } R_1(x;s,t) = \int_0^1 \dot \gamma(x,s+u(t-s)) \, du 
\end{equation*}
and $\lim_{t \to s} R_1(x;s,t) = \dot \gamma(x;s)$. Assume there is a critical point at $s=0$, $\gamma(x;0) = \dot \gamma(x;0) = 0$, hence
\begin{equation}
     \label{eq:resolution}
\gamma(x;0) = 0 \quad \Rightarrow \quad \gamma(x;t) = 
t \ R_1(x;0,t) \quad \text{with $\lim_{t \to 0} R_1(x;0,t) = 0$} \, ,    
\end{equation}
and
\begin{equation*}
    \dot \gamma(x;t) = \velocity \gamma(x;t) \ t \ R_1(x;0,t) \, 
\end{equation*}
where $\velocity \gamma(t)=s(t)$ (see Definition~\ref{def_Velocity}).

If the model $\gamma$ is twice differentiable, then
\begin{equation*}
    \velocity \gamma(x;t) \gamma(x;t) = \dot \gamma(x;t) = R_1(x;0,t) + t \pderivby t R_1(x;0,t) \ , \quad \lim_{t \to s} \pderivby t R_1(x;{s},t) = \frac12 \ddot \gamma(x;s) \, .
\end{equation*}
Or, we can use the expansion
\begin{equation*}
    \gamma(x;t) = \gamma(x;s) + (t-s) \dot \gamma(x;s) + (t-s)^2 R_2(x;s,t) \quad \text{with $\lim_{t \to s} R_2(x;s,t) = \frac12 \ddot \gamma(x;s)$} \, ,
\end{equation*}
to get, in the singularity at $s=0$ that
\begin{equation*}
    \gamma(x;t) = t^2 \ R_2(x;0,t) \, ,
\end{equation*}
hence
\begin{equation*}
    \dot \gamma(x;t) = \velocity \gamma(x;t) \ t^2 \ R_2(x;{0,}t) \, .
\end{equation*}

\begin{ex}[Entropy production] \label{ex:entropy-1} 

The entropy production rate is the first derivative with respect to time of $l(p)=-p\log p$. In the computation of entropy production along a statistical model, one needs to check the critical points of $l(p)$. 
Let the probability model $\gamma(t)$ have an isolated zero at $(x,0)$. We have $- \derivby t l(\gamma(x;t)) = (\log(\gamma(x;t)) + 1) \dot\gamma(t)$ on non-critical points. Otherwise, assuming $\gamma \in C^2$,
\begin{equation*}
    t^{-1}l(\gamma(x;t)) = t^{-1}(t^2 R_2(t) \log(t^2 R_2(t))) = t^{-1}(l(t^2)R_2(t)+t^2l(R_2)) \to 0 \ \text{ for } t\mapsto 0\, .
\end{equation*}
 Hence, the entropy function is differentiable at $t=0$.
 \end{ex}
 
\begin{ex}[Square root transformation] The transformation $p \mapsto 2 \sqrt p = \rho$ maps the probability simplex to the sphere of radius 2, which is a Riemannian manifold. It is a way to induce a Riemannian structure on the probability simplex. The statistical application of such a transformation is discussed, for example, in~\cite{amari:2009}. The derivative on the sphere is related to the score on the simplex by
 \begin{equation*}
  \derivby t \rho(t) = \derivby t 2 \sqrt {\gamma(t)} = \frac {\velocity \gamma(t) \gamma(t)} {\sqrt{\gamma(t)}} = \frac12 \velocity \gamma(t) \rho(t) \, .  
 \end{equation*}
\end{ex}

\subsection{The equation of velocity}

Next, we express the properties of the score in algebraic terms by  means of 
Theorem~\ref{prop:score}.  
Given a differentiable curve $\gamma\in \Delta(\Omega)$, for each $t$ its first derivative is the product of an element in the space of contrasts and the curve itself.

In the ring $\reals^\Omega$, the set of all vectors of the form $s\gamma$ is the ideal generated by $\gamma$. Hence, the derivative satisfies
\begin{equation*}
    \dot \gamma \in C(\Omega) \cap \idealof{\gamma} \, 
\end{equation*}
where $\idealof {\gamma(t)}$ is loosely defined as the set of all functions from $\Omega$ to $\reals$ which are multiples of $\gamma(t)$, for any given $t$. This property implies that the support of the derivative is equal to, or contained within, the support of the model curve $\gamma$.
Conversely, the set of all scores $s$ expressing the same velocity is defined by the equations
\begin{equation*}
    \sum_{x\in \Omega} \gamma(x) = 1,\  \sum_{x\in \Omega}  \dot\gamma(x) = 0, \  \dot \gamma(x) = s(x) \gamma(x) \text{ for all } x\in \Omega \, .
\end{equation*}

To solve the singularity in Equation~\eqref{eq:score}, given a $\gamma \in \Delta(\Omega)$, let $S(t)$ be the indicator function of the support of $\gamma$ at time $t$. For each $t$, $S(t):\Omega \rightarrow \reals$ is such that
\begin{equation}
    \label{eq:support}
S^2 = S \ , \quad  S\gamma = \gamma  \ , \quad
S \in \idealof \gamma \, . 
\end{equation}
For each $S$ such that $ S^2= S$, the set of all $\gamma \in \Delta(\Omega)$ such that Equation~\eqref{eq:support} holds coincides with the open simplex on $S$.
Equation~\eqref{eq:support}
shows that for a smooth curve in $\Delta(\Omega)$ it holds $\dot \gamma(t) \in \idealof {\gamma(t)}$. Indeed, this follows by considering the real-valued vector space isomorphism between the set of real-valued functions defined on $\Omega$ and the ring of polynomials in as many indeterminates as the cardinality of $\Omega$, modulo the ideal generated by the sum-to-one condition.

 In particular, we can generalise the definition of velocity as $\velocity \gamma = \dot \gamma / \gamma$ to the boundary of the simplex as follows.
 
 \begin{defn}[Score]\label{def_Velocity}
 If $ S(t)$ is the indicator function of the support of $\gamma$, 
 then the score $\velocity\gamma(t)$ is defined by any curve $t \mapsto (\gamma(t),\velocity \gamma(t))$ such that  
  \begin{equation}
     \velocity \gamma(t) \gamma(t) 
     - \dot \gamma(t) S(t)  = 0 \, .
 \end{equation}
 \end{defn}

The difference with the standard case as discussed, for example, in~\cite{chirco|pistone:2022}, is that we allow for variable support and, consequently, for a variable dimension of the tangent space. In the case of the open simplex, the score was $\velocity \gamma(t) = \derivby t \log \gamma(t)$. The $*$-notation has been introduced in the literature about non-parametric information geometry, for example the tutorial~\cite{pistone:2020-NPCS}. In simple words, 
the Fisher geometry is constructed using the Fisher score, that is, the logarithmic derivative of $\gamma$ in the open simplex; on the boundary the Fisher score is defined only implicitly via the algebraic equation in Definition~\ref{def_Velocity}.

Moreover, the expected value of $s(t) = \velocity \gamma(t)$ with respect to $\gamma(t)$ is zero for all $t$, indeed 
\begin{equation*}
    \expectat {\gamma(t)} {s(t)} = \sum_{x\in \Omega} s(x;t)\gamma(x;t) = \sum_{x \in \Omega} \dot\gamma(x;t) = 0 \, .
\end{equation*}

\subsection{Examples}

In the following examples we consider models described by polynomial equations. By applying the differentiation rule in Equation~(\ref{eq:scoreA}), according to which the velocity is given by the product of the curve and the score, we directly obtain the corresponding equations of the model in the tangent bundle.

\begin{ex}[a line in the two-dimensional simplex] 
Let $\Omega=\{1,2,3\}$. The following curve gives a probability function in $\Delta(\Omega)$ for  each $t\in [0,1/2] $
\begin{equation*}
  p_1 = t \qquad p_2 = t \qquad p_3 = 1 - 2 t  
\end{equation*}
which hits a one-dimensional face of $\Delta(\Omega)$ at $t=1/2$ and a vertex of  $\Delta(\Omega)$  at $ t=0$. 
Its first derivative in $t$ is $(1,1,-2)$. 
The implicit polynomial form of this one dimensional statistical model is
\begin{equation}\label{ex:line3}
p_1-p_2= 0 =p_3+2 p_1-1
\end{equation}
To the model Equations~\eqref{ex:line3}
we add the tangent conditions $\dot p_i = s_i\ p_i$ for $i=1,2,3$, the sum-to-one constraint and the constraint on the velocities
\begin{equation*}
 p_1-p_2= 0 =p_3+2 p_1-1 \qquad
\dot p_i = s_i\ p_i \text{ for } i=1,2,3 \qquad
p_1+p_2+p_3-1 = 0 = \dot p_1 + \dot p_2 + \dot p_3
\end{equation*}
We have considered the ideal in $\mathbb R[p,\dot p,s]$ generated by 
\begin{equation*}
p_1-p_2, \quad p_3+2 p_1-1 ,\qquad
\dot p_i - s_i\ p_i \text{ for } i=1,2,3   \qquad
p_1+p_2+p_3-1 ,\quad \dot p_1 + \dot p_2 + \dot p_3
\end{equation*} 
In the differential ring with derivation rules $\dot p_i=s_i p_i$, we retrieve
$\dot p_1-\dot p_2, \ \dot p_3+2\dot p_1$ which are as expected because on the model the equalities
$\dot p_1-\dot p_2 = 0= \dot p_3+2\dot p_1$  also hold when $t=0$ or $1/2$.
Written in the probability and score variables the model is given by the two polynomial equations 
$s_1p_1-s_2p_2 = s_3p_3 + 2s_1p_1=0$ 
which provide the implicit expression of the tangent bundle.
\end{ex}

\begin{ex}[$2 \times 2$ contingency table] \label{ex_indep} 
The computations above can be applied to two-parameter statistical models such as the $2 \times 2$ independence model, which we use here as preparation for Example~\ref{ex_marghomo}.
Let $p_{ij}$, $i,j=1,2$ denote the cell probabilities and consider the ring $\mathbb{R}[p_{ij},\dot{p}_{ij},s_{i,j} \ : \ i,j=1,2]$. We stress that the $p$'s are probabilities and the $\dot p$'s are contrasts, hence
\begin{gather}
    p_{11} + p_{12} + p_{21} + p_{22} = 1 \label{eq:ex-1} \\ 
    \dot p_{11} + \dot p_{12} + \dot p_{21} + \dot p_{22} = 0 \label{gathercond}
\end{gather}
Under the marginal independence model, we have two types of equations:
\begin{enumerate}
\item a binomial equation for the independence condition 
\begin{equation}\label{mod:indep}
p_{11}p_{22}-p_{12}p_{21}=0
\end{equation}
\item and the equations for the conditions that define the scores
\begin{equation}\label{P:cond}
    \dot{p}_{ij}=s_{ij}p_{ij}  \qquad \qquad i,j=1,2
\end{equation}
\end{enumerate}
For $p$ variables that are differentiable functions of one parameter, differentiating Equation~\eqref{mod:indep} with respect to the parameter and using Equation~\eqref{P:cond}, we obtain
\begin{gather}
\dot{p}_{11}p_{22}+p_{11}\dot{p}_{22}- \dot{p}_{12}p_{21}-p_{12}\dot{p}_{21} = 0
\nonumber \\ 
s_{11}p_{11}p_{22}+p_{11}s_{22}p_{22}- s_{12}p_{12}p_{21}-p_{12}s_{21}p_{21} = (s_{11}+s_{22})p_{11}p_{22}- (s_{12}+s_{21})p_{12}p_{21} = 0 \label{eq:ex-3}
\end{gather}
Using Equations~\eqref{gathercond} and~\eqref{P:cond}, we add
the Fisher score conditions and obtain
\begin{equation}
s_{11}p_{11} + s_{12}p_{12} + s_{21}p_{21} + s_{22}p_{22} = 0 \label{eq:ex-4}
\end{equation}
Equations~\eqref{eq:ex-3} and~\eqref{eq:ex-4} express the model in the probability and score variables. 
In the open simplex, since $p_{11}p_{22}=p_{12}p_{21} \neq 0$ we get a simple equation involving only the score variables
\begin{equation*}
s_{11}+s_{22}-s_{12}-s_{21} = 0 \ . 
\end{equation*}
This is consistent with the classical result that, under independence, scores are linearly related.
\end{ex}

\begin{ex}($2\times2$ table with independence and marginal homogeneity).
\label{ex_marghomo}
We modify the independence model in Example~\ref{ex_indep} by adding the condition $p_{12}-p_{21}=0$, which in explicit form is 
\begin{gather*}
    p_{11} = (1-\theta)^2 \\
    p_{12} = p_{21} = \theta(1-\theta)\\
    p_{22} = \theta^2
\end{gather*}
Thus we work with three probabilities using the equations
\begin{enumerate}
\item the model 
\begin{equation}\label{mod:indep:s}
p_{11}p_{22}-p_{12}^2=0
\end{equation}

\item scores conditions
\begin{equation}\label{P:cond:s}
    \dot{p}_{ij}=s_{ij}p_{ij}  \qquad \qquad (i,j) \in \{(1,1),(1,2),(2,2)\}
\end{equation}
\end{enumerate}
For $p$ variables that are differentiable functions of one parameter, differentiating Equation~\eqref{mod:indep:s}, we have
\begin{equation*}
\dot{p}_{11}p_{22}+p_{11}\dot{p}_{22}- 2p_{12}\dot{p}_{12} = 0
\end{equation*}
and using Equation~\eqref{P:cond:s}, we obtain the model equations in the $p$ and $s$ variables
\begin{equation*}
s_{11}p_{11}p_{22}+p_{11}s_{22}p_{22}- 2s_{12}p_{12}^2 = 0
\end{equation*}
\begin{equation*}
(s_{11}+s_{22})p_{11}p_{22}- 2s_{12}p_{12}^2 = 0
\end{equation*}
In the open simplex, since $p_{11}p_{22}=p_{12}^2$ we get a linear condition on the scores
\begin{equation*}
s_{11}+s_{22}-2s_{12} = 0
\end{equation*}
\end{ex}
This fact is generalised in the following results for models that can be expressed as binomial equations. 

\begin{lem}
Let $p_i$ be indeterminates and $\dot p_i=s_i \ p_i$ for all $i$. 
Given a monomial $p^\alpha$ then its derivative is $\langle \alpha, s \rangle p$.
\end{lem}
\begin{proof}
This follows from the Leibniz product rule for derivatives.
\end{proof}

\begin{thm} \label{bin:eq}
Given a binomial equation $p^\alpha-p^{\beta}=0$, then on the open simplex:
\begin{equation}\label{s:cond}
\langle (\alpha - \beta), s \rangle = 0
\end{equation}
\end{thm}
\begin{proof}
    The derivative of $p^\alpha-p^{\beta}=0$ must be zero and the previous lemma implies that the linear equation on $s$ must be true on the open simplex.
\end{proof}

As a corollary, for the special case of binomial models Theorem~\ref{bin:eq} gives an alternative set of equations involving only the indeterminates $p$ and $s$ (while $\dot{p}$ has been eliminated).

\section{Statistical bundle} \label{sec:statistical-bundle} 

In this section we redefine the bundle of velocities following the notion of the Fisher score as a velocity, as discussed in Section~\ref{sec:score}.
In Theorem~\ref{thm:statistical-bundle}  we show the consistency of the new notion. Exponential families provide a natural mapping from the fibers of the bundle to the probability simplex. We use one-dimensional exponential families in the form
\begin{equation*}
    \gamma(t) = \euler^{ts - \psi(t)} \  p \ , \quad \expectat p s = 0 \ , \quad \psi(t) = \log \expectat p {\euler^{ts}} \, .
\end{equation*} 
See~\cite{brown:86} on exponential families.~\cite{pistone|sempi:95} used the $K_p$-notation to denote the cumulant function. That is,
\begin{equation*}
    q = \euler^{u - K_p(u)} \  p \ , \quad \expectat p u = 0 \ , \quad K_p(u) = \log \expectat p {\euler^u} \, .
\end{equation*}

\begin{defn} \label{def_statEscorebundles}
Let $L(\Omega)=\reals^\Omega$ be the vector space of real-valued functions on $\Omega$.
\begin{enumerate}
\item 
The statistical bundle is the set of pairs
\begin{equation} \label{eq:statistical-bundle}
\expbundleat \Omega = \setof{(p,s)}{p \in \Delta(\Omega), s \in L(\Omega) \text{ such that }\expectat p s = 0} \, .
\end{equation}

The fiber at $p$ is $\expfiberat p \Omega = L_0(p) = \setof{s \in L(\Omega)}{\expectat p s = 0}$. Each fiber is endowed with the bilinear form $\scalarat p s v = \expectat p {s v} = \covat p s v$ for $s, v\in L_0(p)$.
\item 
The score bundle is the union over $p\in \Delta(\Omega)$ of the $p$-score fiber, defined as  
\begin{equation*} 
\setof{s \in L(\Omega)}{ \text{there exists } \dot p\in C(\Omega) \text{ such that }\dot p = s \ p } \ . 
\end{equation*}
\end{enumerate}
\end{defn}
Note that Theorem~\ref{theorem_contrast_bundle} and Definition~\ref{def_statEscorebundles} are akin: one is formulated in algebraic–geometric language and the other in probabilistic–statistical language.
Theorem~\ref{thm:statistical-bundle} below shows that the statistical bundle is indeed a tangent bundle.

\begin{thm}\label{thm:statistical-bundle} 
The statistical bundle is equal to the score bundle.
\end{thm}
\begin{proof}
If $\dot p = sp$, then $\expectat p s = \sum_x \dot p_x = 0$. Conversely, let $p \in \Delta(\Omega)$, and let $s \in L(\Omega)$ be such that $\expectat p s = 0$. Define the exponential family
\begin{equation*}
    \gamma(t) = \euler^{ts - \psi(t)} \  p \, .
\end{equation*} 
Then, 
\begin{equation*}
    \left. \derivby t \gamma(t) \right|_{t=0} = \left. \left(s - \derivby t \psi(t)\right) \gamma(t) \right|_{t=0} = 
    \left. \left( s -\expectat {\gamma(t)} s \right) \gamma(t) \right|_{t=0} =\  s\ p \, .
\end{equation*}
The last equality holds because $\gamma(0)= p$ and $\expectat {p} s =0$.
\end{proof}

Theorem~\ref{thm:statistical-bundle} shows that, for any $p \in \Delta(\Omega)$, the vector space $L_0(p) = \setof{u \in L(\Omega)}{\expectat p u = 0}$ equals the fiber of all possible Fisher scores at $p$. The space $L_0(p)$ does not represent \emph{uniquely} all scores because if $u,v \in L_0(p)$ and $(u-v)p=0$, that is, $u$ and $v$ are equal on $\supp p$ but differ otherwise, they represent the same score. To ensure uniqueness consider the equivalence relation
\begin{equation} \label{eq:equivalence}
    (p,u), (p,v) \in \expbundleat \Omega \ , \quad (p,u) \sim (p,v) \Leftrightarrow (u-v)p = 0 \ ,
\end{equation}
that is, $u = v$ on $\supp p$.
\begin{defn}
The \emph{reduced statistical bundle} is the quotient of the statistical bundle for the relation $\sim$,
\begin{equation} \label{eq:red-statistical-bundle}
\rexpbundleat \Omega = \setof{(p,u)}{p \in \Delta(\Omega),  \ u \in \frac{L_0(p)}{\sim} = L_0(\supp p,p) \ } \ ,
\end{equation}
where $L_0(\supp p,p) = \setof{u \in L(\supp p)}{\expectat p u = 0}$ is a Hilbert space for $(u,v) \mapsto \scalarat p u v$.
\end{defn}

The statistical bundle inherits the vector space structure over the reals from that of $L_0(p)$. Moreover, it is described by the semi-algebraic conditions,
\begin{equation*}
 \sum_x p_x = 1, \ \sum_x u(x)p_x = 0, \ p_x \geq 0 \, .   
\end{equation*}

The following proposition gives an interpretation of the score as a velocity in the statistical bundle.
\begin{prop}
Let $t \mapsto \gamma(t)$ be a $C^1$ curve in the probability simplex. Let $F \colon \reals^\Omega \to \reals$, $F \in C^1$ be an implicit relation on the model, that is, $F(\gamma(t)) = 0$ for all $t$. Then for the gradient of $F$, $\nabla F$, it holds
\begin{equation*}
    \covat {\gamma(t)} {\nabla F(\gamma(t)}{\velocity \gamma(t))}  = 0 \, .
\end{equation*}
\end{prop}
\begin{proof}
We have $\expectat{\gamma(t)}{\velocity \gamma(t)} = 0$, hence
    \begin{equation*}
        0 = \derivby t F(\gamma(t)) = \sum_x \pderivby {\gamma(x)} F(\gamma(x;t)) \velocity \gamma(x;t) \gamma(x;t) = \expectat {\gamma(t)} {\nabla F(\gamma(t)) \velocity \gamma(t)} = \covat{\gamma(t)}{\nabla F(\gamma(t))}{\velocity \gamma(t)} \, .
    \end{equation*}
\end{proof}

With Definition~\ref{def:lagrangian} below we introduce a differentiable structure on the statistical bundle by following the standard formalism of mechanics (\cite{landau|lifshits:1976} and~\cite{chirco|malago|pistone:2022}).
\begin{defn}[Statistical Lagrangian] \label{def:lagrangian}
\
\begin{enumerate} 
\item A Lagrangian is a mapping
\begin{equation*}
 \Lambda \colon \expbundleat \Omega \ni (p,u) \mapsto \Lambda(p,u) \in \reals \, .
 \end{equation*}
\item 
The Lagrangian is statistical if it is compatible with the equivalence relation of Equation~\eqref{eq:equivalence}, that is,
\begin{equation*}
    (u - v)p=0 \quad \text{ implies } \quad \Lambda(p,u) = \Lambda(p,v) \, .
\end{equation*}
\end{enumerate} 
\end{defn}

If $t \mapsto (\gamma(t),u(t)) \in \expbundleat \Omega$ is a differentiable curve and $\Lambda$ is a differentiable function on an open set containing the statistical bundle, then
\begin{equation*}
    \derivby t \Lambda(\gamma(t),u(t)) = \expectat{\gamma(t)}{\nabla_1 \Lambda(\gamma(t),u(t)) \velocity \gamma(t)} + \nabla_2 \Lambda(\gamma(t),u(t)) \ \dot u(t)) \,
\end{equation*}
where $\nabla_1$ denotes the partial derivative with respect to the first argument of $\Lambda$,  and  $\nabla_2$ denotes the partial derivative with respect to the second argument.
 
\begin{ex}
Consider $\Lambda (p,u) = K_p(u) = \log \expectat p {\euler^u}$ with $(p,u) \in \expbundleat \Omega$. If $(u-v) \ p=0$, then $\Lambda(p,u) = \Lambda(p,v)$, hence the cumulant functional is a statistical Lagrangian. If $t \mapsto (\gamma(t),u(t))$ is differentiable, then
\begin{align*}
    \derivby t K_{\gamma(t)}(u(t)) & = {\expectat {\gamma(t)} {\euler^{u(t)}}}^{-1} \sum_x \derivby t \gamma(x;t) \euler^{u(t)} \\ &= \euler^{- K_{\gamma(t)}(u(t))}
    \sum_x (\dot \gamma(x;t) \euler^{u(x;t)} + \gamma(x;t) \euler^{u(x;t)} \dot u(x;t)) \\ &= \expectat {\gamma(t)}{\velocity \gamma(t) \euler^{u(t)- K_{\gamma(t)}(u(t))} + \euler^{u(t)- K_{\gamma(t)}(u(t))} \dot u(t)} \, .
\end{align*}

In the equation above, the factor $\gamma^u(t) = \euler^{u(t)- K_{\gamma(t)}(u(t))} \  \gamma(t)$ is a new model. It holds
\begin{equation*}
\derivby t K_{\gamma(t)}(u(t)) = \expectat{\gamma^u(t)}{\velocity \gamma(t) + \dot u(t)} \ .
\end{equation*}
The algebraic properties of the derivatives of the cumulant function have been discussed in~\cite{pistone|wynn:99} and~\cite{pistone|wynn:2006}. In the case $u = \velocity \gamma$, we have a notion of acceleration. We do not discuss this here and refer to the tutorial~\cite{chirco|pistone:2022}.
\end{ex}

Given $p$, consider the vector space $J_p$ of all $k \in  \mathbb R^\Omega$ such that $k \ p=0$. Then $u$ and $v$ represent the same score if, and only if $u-v \in J_p$. 
From an algebraic perspective, the reduced statistical bundle in Equation~\eqref{eq:red-statistical-bundle} is a semi-algebraic set. Indeed, let $\reals[p,u]$ be the ring of polynomials with indeterminates $p_x,u(x), x \in \Omega$. The statistical bundle is given by the zero set of the polynomials
\begin{gather*}
    \sum_x p_x - 1  \quad \text{and} \quad
    \sum_x u(x)p_x 
\end{gather*}
and by the inequalities $p_x \geq 0$. Note that $J_p$ is an ideal, so that
\begin{gather*}
\left\{ p \in \mathbb R^\Omega: 
    \sum p_x-1=0, \ \sum u(x)p_x=0, p_x\geq 0, \ u(x)-v(x)\in J_p \right\} \\
 =  \ V_{p\geq 0 }\left(\sum p_x-1=0, \ 
    \sum u(x)p_x=0,\ 
    u(x)-v(x)\in J_p \right) \ .
\end{gather*} 

\section{Natural gradient} \label{sec:natgradient}

The extension of the Fisher score to the closed simplex allows for the generalisation of well-known differential computations in mathematical statistics. A notable example is the Cramer-Rao inequality,~\cite{rao:45}, which provides a statistical bound on the variance of any estimator with respect to a given one-parameter model $\gamma(t)$. The position $\dot \gamma = \velocity \gamma \gamma$ in Theorem~\ref{prop:score} is the key step in the derivation of the inequality and, more generally, the computation of the natural gradient of~\cite{amari:1998natural} and~\cite{amari|nagaoka:2000}.

\begin{prop} \label{prop:CR}
Let $t \mapsto \gamma(t)$ be a differentiable curve in the probability simplex and $g$ a random variable on $\Omega$. 
Then
\begin{equation*}
    \left( \frac d {dt} \expectat{\gamma(t)}{g}\right)^2 \leq \varat{\gamma(t)} {\velocity \gamma(t)} 
    \
    \varat{\gamma(t)}{g} \, .
\end{equation*}
\end{prop}

\begin{proof}
Let $G(t) = \expectat{\gamma(t)}{g}$ be the expected value of $g$. Then
\begin{equation*}
\dot G(t) = \sum_x \dot \gamma(x;t) g(x) = \sum_x \velocity \gamma(x;t) g(x) \gamma(x;t) = \sum_x \velocity \gamma(x;t) (g(x) - \expectat{\gamma(t)}{g})\gamma(x;t) = \covat {\gamma(t)}{\velocity \gamma(t)}{g}  \, .
\end{equation*}
This follows from the fact that $\dot\gamma(x;t)$ is in the tangent bundle:
\begin{equation*}
0 = \sum_x \velocity \gamma(x;t)  \gamma(x;t) = \sum_x \dot \gamma(x;t) 
\end{equation*}
and the result follows by applying the Schwartz inequality.
\end{proof}

More generally, let $G \colon \Delta(\Omega) \to \reals$ be a statistical estimator such that $G(\gamma)$ is differentiable. Then by Proposition~\ref{prop:CR} we can write:
\begin{align*} 
     \derivby t G(\gamma(t)) & =  \sum_x \frac {\partial \, G(\gamma(t))} {\partial \gamma(x)} \,  \dot \gamma(x,t) = \sum_x \frac{\partial \, G(\gamma(t)) } {\partial \gamma(x)} \,  \velocity \gamma(x,t) \ \gamma(x,t) 
     \\ & = \expectat {\gamma(t)} {(\nabla G(\gamma(t)) - \expectat {\gamma(t)}{\nabla G(\gamma(t)) }) \, \velocity \gamma(t)} \\
     & = \scalarat{ {\gamma(t)} }{{\nabla G(\gamma(t)) - \expectat {\gamma(t)}{\nabla G(\gamma(t)) }}}{{\ \velocity \gamma(t)}}
     \, . 
\end{align*}
where the last equality follows from the fact $\scalarat \gamma u v = \expectat \gamma {uv}$, for any $u,v$ such that $\expectat \gamma u = \expectat \gamma v = 0$. This leads to the following definition.

\begin{defn}[Natural gradient] 
Let $G \colon \Delta(\Omega) \to \reals$ be such that $t \mapsto G(\gamma(t))$ is infinitely differentiable for all  differentiable
curves $\gamma$.
The natural gradient of $G$, $\Grad G$, is a section of the statistical bundle, $(\gamma,\Grad G(\gamma)) \in \expbundleat \Omega$ such that
\begin{equation}
    \label{eq:natural-gradient}
    \derivby t G(\gamma(t)) = \scalarat {\gamma(t)} {\Grad G(\gamma(t)) \ }{ \ \velocity \gamma(t)}
\end{equation}
\end{defn}

In this definition we used the general concept of gradient as representation of the derivative operator in an inner product, the inner product being the Fisher metric on the statistical bundle. That is, the natural gradient is the Riemannian gradient with respect to the Fisher metric. Amari originally derived his notion of natural gradient from a metric argument involving a computation of divergence variation along a curve, see \cite{amari:1998natural}. Such a derivation, in the literature, is frequently taken as a definition because of its intuitive appeal. However, Amari himself has always used his signature term ``natural gradient'' as a replacement of the standard term. See \cite[Ch.~12.1.2]{amari:2016}. The term natural gradient is also frequently used in connection with the fact that its computation in parametric models requires a matrix inversion, a costly, but useful operation. 

In the subsection below, through examples, we show that for monomial models these operations are extended to the closed simplex via algebra.    

\subsection{Examples} 
\begin{ex} [Follows from Example \ref{ex:entropy-1}] \label{ex:entropy-2}
  \begin{figure}[t]
    \centering
    \begin{tabular}{cc}
    \includegraphics[width=.4\linewidth]{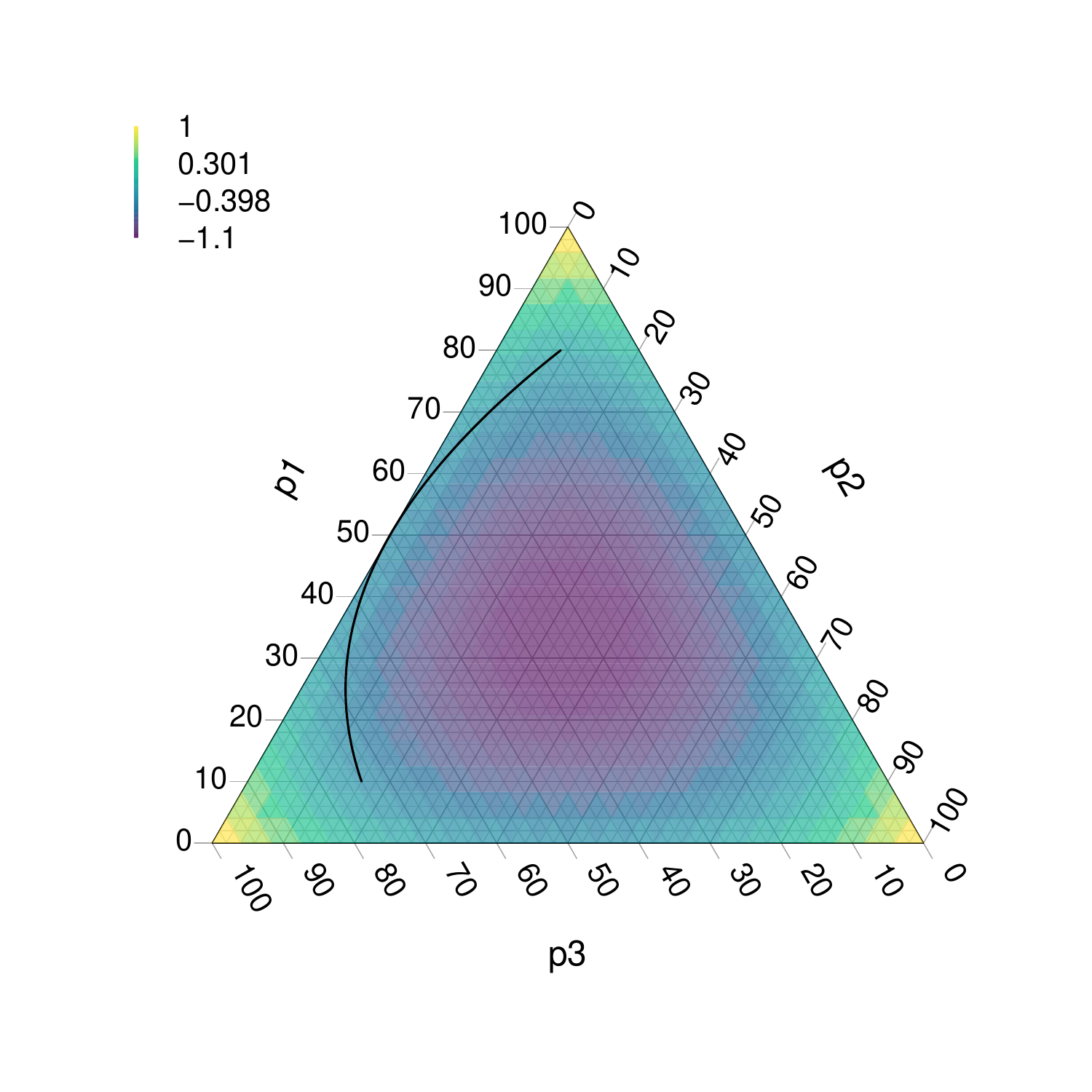}    & \includegraphics[width=.4\linewidth]{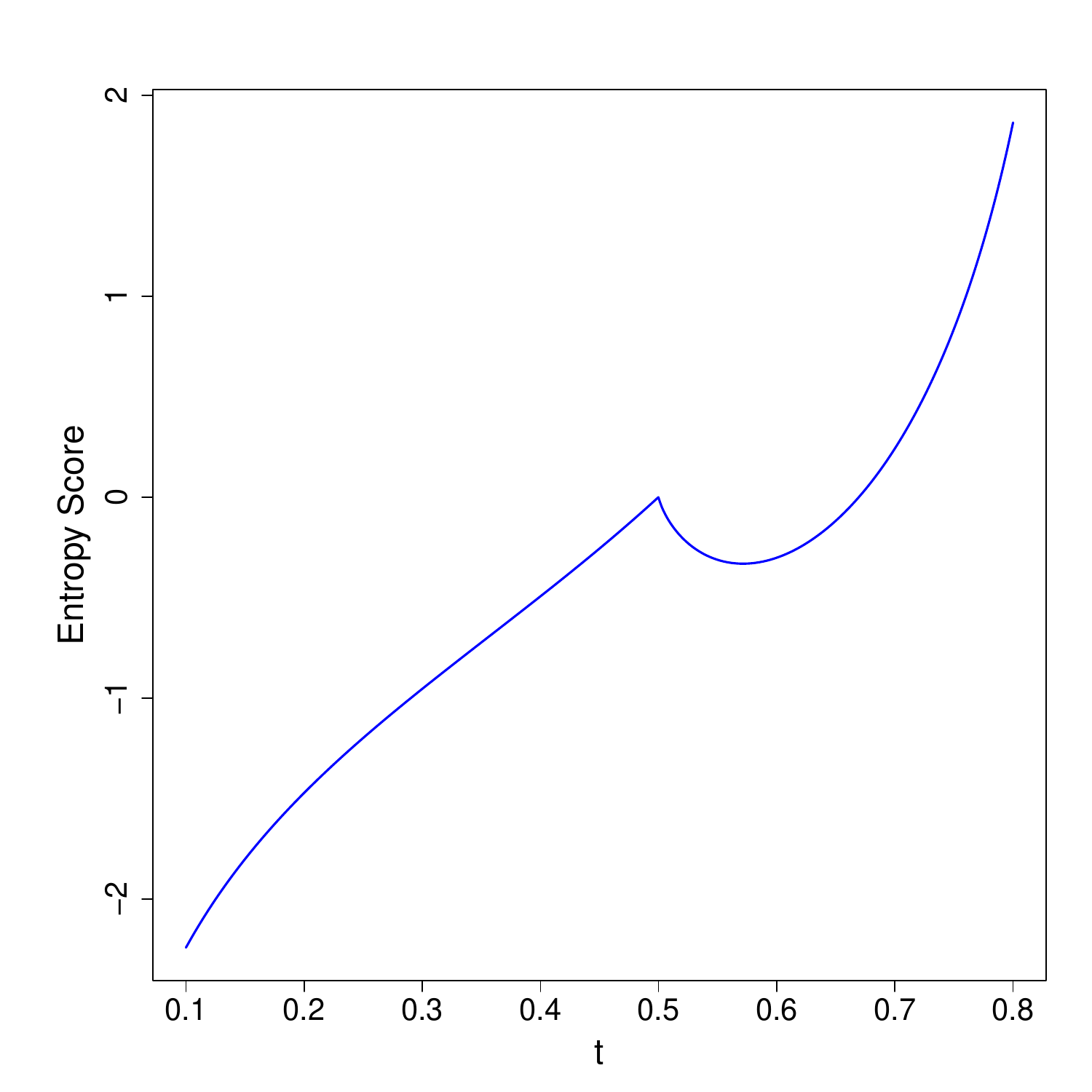}
    \end{tabular}
\caption{Illustration of Example \ref{ex:entropy-2}: entropy with a defective support. Left: entropy heatmap on the simplex, with the model curve shown in black. Right: entropy production along the curve as a function of $t$.}
    \label{fig:entropy}
  \end{figure}    
 Assume $\gamma$ is infinitely differentiable and set $ l(\gamma) = - \gamma \log\gamma  $. Then as shown in Exmple \ref{ex:entropy-1} the entropy $l$ is differentiable and the entropy production is
    \begin{equation*}
       \derivby t l(\gamma(x;t)) = -(\log(\gamma(x;t))+1) \ \velocity \gamma(x;t) \ \gamma(x;t) \, .
    \end{equation*}
Here we compute the natural gradient of the entropy along a curve $\gamma$, where $\gamma(t) \in\Delta(\Omega)$ and the entropy is defined as $H(\gamma(t)) = - \sum_{x\in \Omega} \gamma(x;t) \log \gamma(x;t)$ for all $t$. 

Summing over $x$, it holds
    \begin{align*}
\derivby t H(\gamma(t))  & = - \sum_x \left(\log(\gamma(x;t))+1\right) \ \velocity \gamma(x;t)  \ \gamma(x;t) 
= \expectat {\gamma(t)}{-(\log \gamma(t)+1) \velocity \gamma(t)} \\
& = \expectat {\gamma(t)}{\left(-(\log \gamma(t)+1) - \expectat {\gamma(t)} {-(\log \gamma(t)+1)}\right) \velocity \gamma(t)} 
\\ &=
\scalarat{\gamma(t)}{-\log \gamma(t) - H(\gamma(t)) \ }{\ \velocity \gamma(t)} 
    \end{align*}
Where we used that 
$(\gamma , -\log \gamma(t) - H(\gamma(t))$ and $ (\gamma , \ \velocity \gamma ) \in \expbundleat \Omega$ and 
$\expectat {\gamma(t)}{ \expectat {\gamma(t)} {-(\log \gamma(t)+1)} \velocity \gamma(t)} = \expectat {\gamma(t)} {-(\log \gamma(t)+1)}
\expectat {\gamma(t)}{ \velocity \gamma(t)}
=0$ because $\expectat {\gamma(t)}{ \velocity \gamma(t)}=0$.

Furthermore, notice that the chain of equalities shows that $\derivby t H(\gamma(t)) $ is equal to the covariance between $-\log \gamma$ and $\velocity \gamma$.

The natural gradient of the entropy is 
\begin{equation*}
    (\gamma \ , \ \Grad H(\gamma) ) 
     =
   (\gamma \ , \ -\log \gamma - H(\gamma))
    \, 
    \in \expbundleat \Omega \, .
\end{equation*}
Note that $\Grad H(\gamma) = 0$ if $-\log \gamma = H(\gamma)$, that is, $\gamma$ is constant on its support. It is the case of maximum entropy.
\end{ex}

\begin{ex} [Follows from Example \ref{ex:entropy-2}] 
Figure~\ref{fig:entropy} provides a numerical example. 
On $\Omega= \set{1,2,3}$, consider the model 
\begin{equation*}
    ].1,.8[ \ni t \mapsto \left(t,\left(t-\frac12\right)^2,\frac34 - t^2\right) \, .
\end{equation*}
with an isolated critical point $\gamma(1/2) = (1/2,0,1/2)$. The score (see Equation~\eqref{eq:score}) is 
\begin{equation*}
    \velocity \gamma(t) = \left(t^{-1},2\left(t-\frac12\right)^{-1},2t\left(t^2-\frac34\right)^{-1}\right) \ , \quad t \neq \frac12 \, .
\end{equation*}
The entropy production is
\begin{equation*}
   - \derivby t H(\gamma(t)) = \log t + 2 \left(t-\frac12\right)\log\left(t-\frac12\right)^2 - 2t \log\left(\frac34-t^2\right) 
 \, .
\end{equation*}
Notice the compensation of the singularity at $t=1/2$.
\end{ex}

\begin{ex} \label{ex:uniformmarg}
We define a one-dimensional model for $2 \times 2$ tables, with probabilities ordered as $(p_{11},p_{12},p_{21},p_{22})$, which is also the smallest example of a discrete copula (i.e., a $2 \times 2$ contingency table with uniform margins). We consider two parametrisations to show their different behaviour in terms of singular points of the entropy. Writing the four cells row-wise, a first parametrisation is
\begin{equation*}
\gamma(t) = \left(t, \frac 1 2 -t, \frac 1 2 -t ,t \right), \qquad t \in \left( 0, \frac 1 2  \right)  \, ,
\end{equation*}
whence
\begin{equation*}
\velocity \gamma(t) = \left(\frac 1 t, \frac {-1} {1/2-t}, \frac {-1} {1/2-t} , \frac 1 t \right) 
\end{equation*}
and 
\begin{equation*} 
  \derivby t H(\gamma(t)) = - 2\log t  + 2 \log \left(\frac 1 2 - t \right)  \, .
\end{equation*}
Such a derivative has a zero in $t=1/4$, corresponding to the table
\begin{equation*}
\begin{pmatrix}
 1 /4 &  1 / 4 \\ 
 1 /4 &  1 / 4 \\
\end{pmatrix}
\end{equation*}
which is the intersection of the model curve $\gamma$ with the independence model. With this parametrisation the entropy derivative diverges when approaching the endpoints of the model curve on the boundary of the simplex.

An alternative parametrisation in the open simplex is given by 
\begin{equation*}
\gamma(s = \left(s^2, \frac 1 2 -s^2, \frac 1 2 -s^2 ,s^2 \right), \qquad s \in \left( - \frac{\sqrt{2}} 2, + \frac{\sqrt{2}} 2 \right)  \, ,
\end{equation*}
The two parametrisations define the same implicit model through the system of linear equations
$ p_{11}-p_{22}=0$,
$p_{12}-p_{21}=0$, and 
$p_{11}+p_{12}-\frac1 2 =0$.
However, with the quadratic parametrisation, we get
\begin{equation*}
\velocity \gamma(s) = \left(\frac 2 s, \frac {-2s} {1/2-s^2}, \frac {-2s} {1/2-s^2} , \frac 2 s \right) 
\end{equation*}

The derivative of the entropy is, for $s \ne 0$,
\begin{equation*}
  \derivby s H(\gamma(s)) = 2s \left( - 2\log s^2  + 2 \log \left(\frac 1 2 - s^2 \right) \right)  \, .
\end{equation*}
with two zeros at $s= \pm 1/2$, both corresponding to the same independence table above. 

Moreover, $s=0$ is a point with zero velocity $\dot\gamma$, corresponding to the following table on the boundary of the simplex:
\begin{equation*}
\begin{pmatrix}
0 &  1 / 2 \\
 1 / 2 & 0 \\
\end{pmatrix} \, .
\end{equation*}
and the limit of $\derivby s H(\gamma(s))$ is $0$ for $s$ approaching $0$. 

In Figure \ref{fig:entr_der} the graphs of the derivative of the entropy under the two parametrisations are displayed. The key point is that the quadratic parametrisation does not change the underlying algebraic curve, nor does it remove the fact that the model reaches the boundary of the simplex. Rather, it changes the way in which the boundary point is reached. In the linear parametrisation, the entropy derivative diverges when the curve approaches the boundary. On the other hand, in the quadratic parametrisation one of the two boundary points becomes regular from the point of view of the entropy variation along the parametrised curve. More precisely, although the entropy gradient is singular at the boundary, the derivative of the entropy along the quadratic parametrization admits a removable discontinuity and can be continuously extended by setting it equal to zero at that point.

\begin{figure}[t]
    \centering
    \begin{tabular}{cc}
    \includegraphics[width=.4\linewidth]{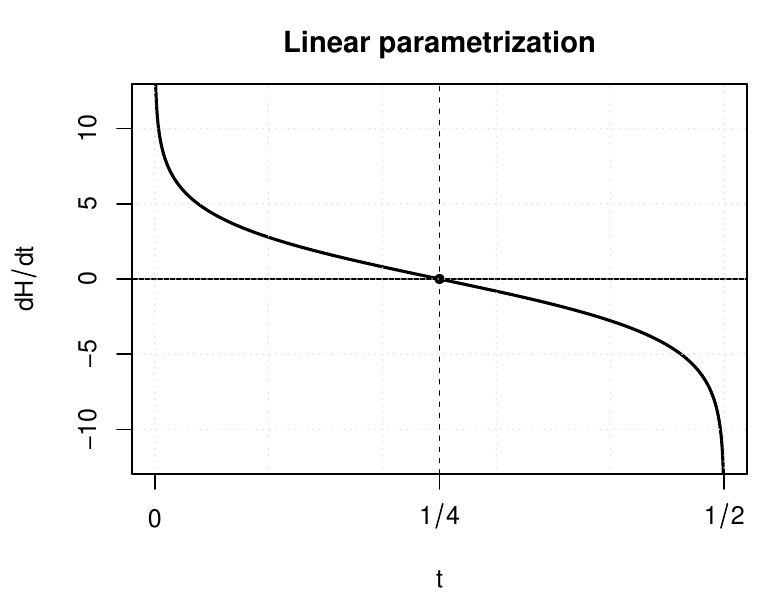}    & \includegraphics[width=.4\linewidth]{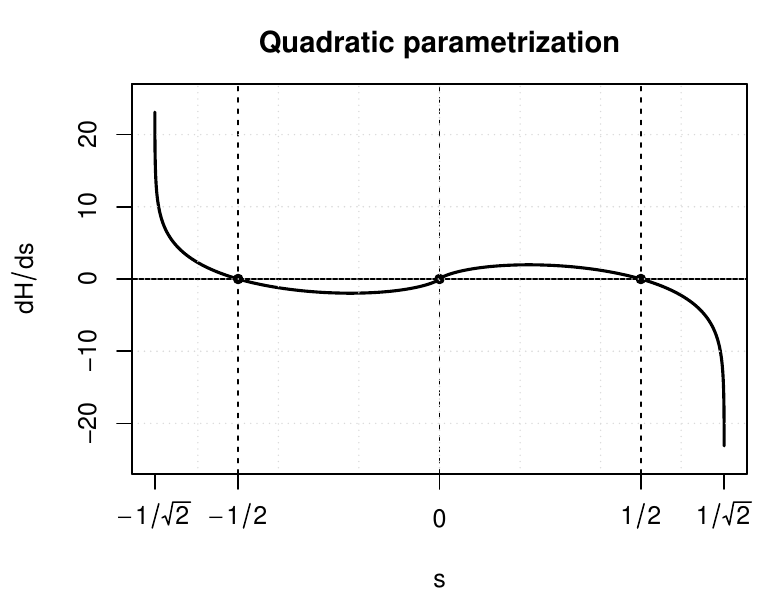}
    \end{tabular}
\caption{Derivatives of the entropy under linear parametrisation (left) and quadratic parametrisation (right) for Example~\ref{ex:uniformmarg}.}
    \label{fig:entr_der}
  \end{figure}    
\end{ex}

\begin{ex}[Boltzmann-Gibbs model] \label{ex:gibbs}
\begin{figure}[t]
    \centering
    \includegraphics[width=.7\linewidth]{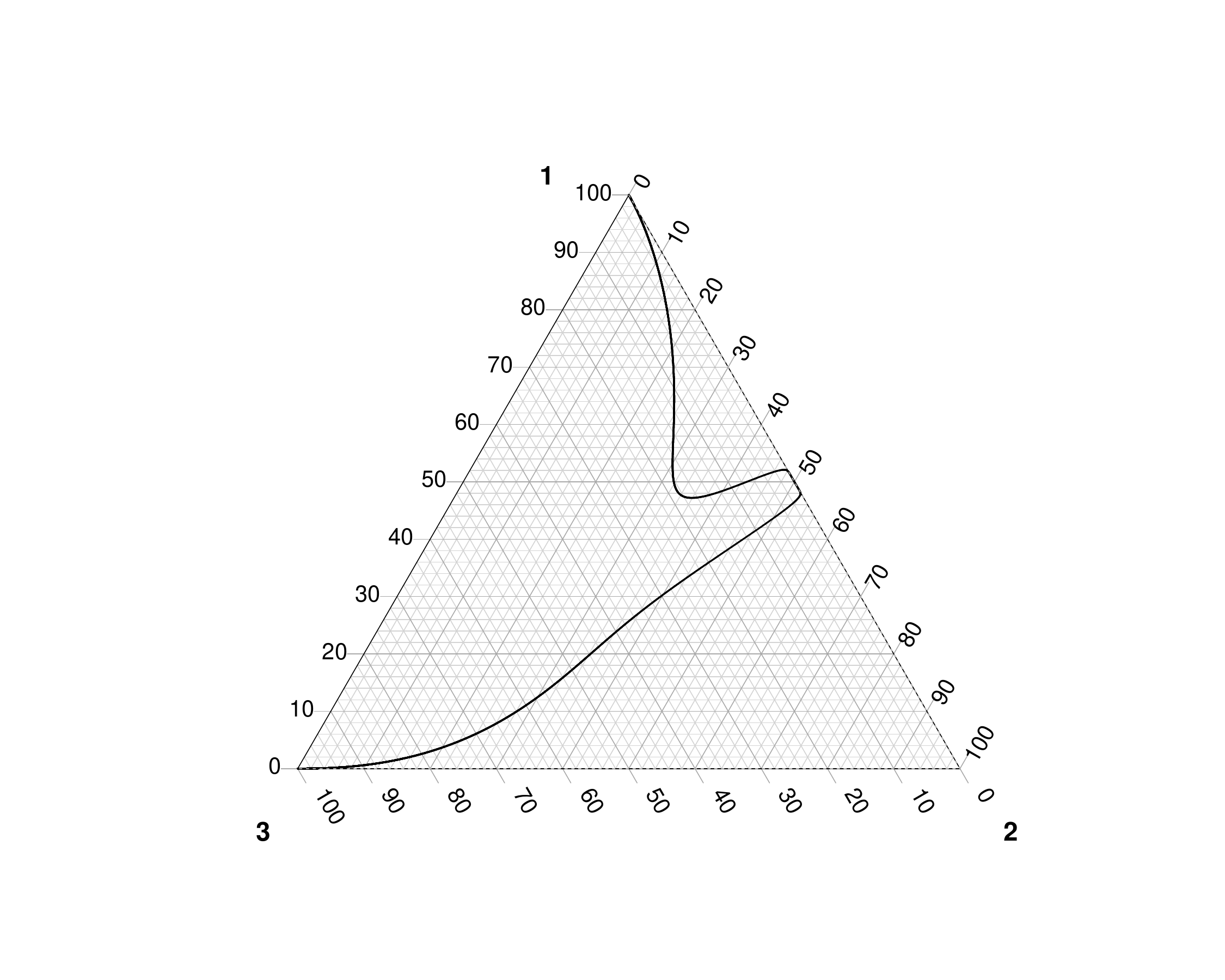}
    \caption{Illustration of Example~\ref{ex:gibbs} with $U=(0,0,1)$ and $V=(0,1,1.8)$. At $\beta=0$ the curve hits the boundary point $(1/2,1/2,0)$ with a contact of exponential order. For $\beta\to\pm\infty$, the curve tends to the vertex associated with the maximum and the minimum of $V$, respectively.
    \label{fig:gibbs}}
  \end{figure}
  Figure \ref{fig:gibbs} is an example of a smooth curve in the probability simplex, illustrating two ways of hitting the boundary. The asymptotic behaviour at vertices 1 and 3 is typical of exponential families, while the tangential hit is a feature of the Boltzmann-Gibbs models in statistical physics (see~\cite[Ch.~3]{landau|lifshits:1980}). Let $U$ be a non-negative random variable with $\min U = 0$ and $V$ any random variable. 
  Let $\#A$ indicate the number of elements of the set $A$ and $(U(x)=0)$ the indicator function of the event $U(x)=0$.
  Consider the model
  on $\Omega=\{1,2,3\}$ 
\begin{equation*}
    \gamma(x;\beta) \propto \expof{ -\beta^{-2} U(x) + \beta V(x)} = (U(x)=0)\expof{\beta V(x)} + (U(x) > 0) \expof{ - \beta^{-2} U(x)+\beta V(x)} \ ,
\end{equation*}
where $\beta$ is a non-zero real-valued parameter. The normalising constant (called the partition function in physics) is 
\begin{equation*}
    Z(\beta) = \sum_{x\in \Omega} \expof{ -\beta^{-2} U(x) + \beta V(x)} = \sum_{U(x)=0}\expof{\beta V(x)} + \sum_{U(x)>0} \expof{ -\beta^{-2} U(x) + \beta V(x)} \, .   \end{equation*}
As the limits
\begin{equation*}
    \lim_{\beta \to 0} \expof{-\beta^{-2}U+\beta V} = (U=0) \quad \text{and} \quad \lim_{\beta\to 0} Z(\beta) = \#\{ U=0 \} \ ,
\end{equation*}
 exist, we can extend the model to $\beta=0$ by
\begin{equation*}
    \gamma(x;\beta) = \begin{cases}
        \expof{ -\beta^{-2} U(x) + \beta V - \log Z(\beta)} & \text{if $\beta \neq 0$, and} \\
        \lim_{\beta \to 0} \gamma(x;\beta) = \displaystyle \frac{(U(x)=0)}{\#\set{U=0}} & \text{if $\beta=0$,}
    \end{cases}
\end{equation*}
so that the curve
\begin{equation*}
    \reals \ni \beta \mapsto \gamma(\beta) \in \simplexof{\Omega}
\end{equation*}
is continuous. The derivative in $\beta$ of the cumulant function $\psi(\beta) = \log Z(\beta)$ is
\begin{equation*}
    \dot\psi(\beta) = \expectat \beta { 2 \beta^{-3} U + V}
\end{equation*}
and
\begin{equation*}
    \lim_{\beta \to 0} \dot \psi(\beta) = \lim_{\beta \to 0} Z(\beta)^{-1}\sum_{x} \euler^{-\beta^{-2}U(x) + \beta V(x)}(2\beta^{-3} U(x) + V(x)) = 0 \ . 
\end{equation*}
Notice that we have used the exponential growth, that is, $\lim _{\alpha \to \infty} \euler^{-\alpha U(x) }\alpha^{3/2} U(x) = 0$ with $\alpha=1/\beta^2$ to show the curve is differentiable with derivative 
\begin{equation*}
    \dot \gamma(x,\beta) = \begin{cases}
    \gamma(x;\beta)\left( (2 \beta^{-3} U(x) + V(x)) - \expectat \beta {2 \beta^{-3} U + V}  \right) & \text{if $\beta \neq 0$, and} \\
    0 & \text{if $\beta=0$,}
    \end{cases}
\end{equation*}

By inspection, we see that the Fisher score is
\begin{equation*}
    \velocity \gamma(x;\beta) = \begin{cases} 2 \beta^{-3}(U(x) - \expectat \beta U) + (V(x) - \expectat \beta V)
 & \text{if $\beta \neq 0$,} \\
 0  & \text{if $\beta=0$ and $U(x)=0$,} \\
  \text{anything} & \text{if $\beta=0$ and $U(x)>0$.}\end{cases}
\end{equation*}

The running variance of the Fisher score, that is, the Fisher information, is
\begin{equation*}
    \expectat {\gamma(\beta)} {{\velocity \gamma(\beta)}^2} = \sum_{\velocity \gamma(x;\beta) >0 } \frac{{\dot \gamma(x;\beta)}^2}{\gamma(x;\beta)} 
\end{equation*}
vanishes for $\beta \to 0$, so that the model is singular (see, for example,~\cite{watanabe:2009}).

The triplot in Figure~\ref{fig:gibbs} is the case where $U=(0,0,1)$ and $V = (0,1,1.8)$. The model is infinitely differentiable in the parameter and cannot be reduced to an algebraic form. If we take $U=0$, the $V$-part is exponential, hence monomial in the parameter $\theta_1=\euler^{\beta}$ as in~\cite[Ch.~5]{pistone|riccomagno|wynn:2001}. If we take $V=0$ with the parametrisation $\theta_2 = \euler^{-1/\beta^2}$, then the model becomes a monomial statistical model. If the two monomial models are denoted $\gamma_i(\theta_i)$, $i = 1,2$, 
our example can be seen as a constrained monomial model
\begin{equation*}
\gamma(\theta_1,\theta_2) \propto \gamma_1(\theta_1)\gamma_2(\theta_2) \ , \quad \text{subject to $\log^2 \theta_1 \log \theta_2 +1$ = 0.} 
\end{equation*}
\end{ex}

\begin{ex} (mixture of densities) Given $p,q \in \Delta(\Omega)$, consider the mixture 
\begin{equation*}
    \gamma(t) = (1-t)p + t q
    \qquad t\in[0,1]
\end{equation*}
In this case, Equation~(\ref{eq:score}) becomes
\begin{equation*}
    q - p = s(t) \left( p + t(q-p) \right) \, .
\end{equation*}
At $t = 0$ we get
$    q - p = s \ p $
and at $t = 1$ we get
$   q - p = s \ q $.

For $t \in ]0,1[$, $\supp \gamma(t) = \supp p \cup \supp q$, which does not depend on $t$,  and the score is
\begin{equation*}
    s(x;t) = \begin{cases}
        \derivby t \log(p_x + t (q(x) - p_x)) = 
         \displaystyle
        \frac{q(x)-p_x}{p_x + t (q(x) - p_x)}
        & \text{if } x \in \supp p \cup \supp q \\
        \text{arbitrary} & \text{otherwise}
    \end{cases} 
\end{equation*}
For $t=1$, $\supp \gamma(1) = \supp q$ and 
the score is 
$\frac{q(x)-p_x}{ q(x) }$ if $x\in \supp q$ and arbitrary otherwise. 
Analogously
$\supp \gamma(0) = \supp p$ and 
the score is 
$\frac{q(x)-p_x}{ p_x }$ if $x\in \supp p$ and arbitrary otherwise. 

\end{ex}

\section{Discussion} \label{sec:final}

The adoption of the affine space defined by the 
displacement $(p,q) \mapsto \log \frac q p - \expectat p {\log \frac q p}$ 
allows for a proper calculus and the discussion of statistical models as differentiable varieties. 
Algebraic statistics is not unfamiliar with affine objects such as Segre varieties and Veronese varieties that have a clear statistical interpretation. Indeed, monomial statistical models and their toric counterparts are represented in the affine coordinates as affine subspaces.
In this paper we argue that the interplay between the differential approach and algebraic statistics not only provides the mathematical tools needed by Fisher statistical methodology, but also allows the inclusion of boundary points.

Indeed, the differential theory applies to the open probability simplex, meaning that all statistical models must have a constant support. This paper provides tools to extend the theory to the closed probability simplex. It is a step towards the development of  analytical and algebraic methods for treating models, such as contingency tables, that are smooth yet exhibit zeros. The key object is the Fisher score. Every smooth 1-dimensional model $\gamma(t)$ has a derivative of the form $\dot \gamma(t) = \velocity \gamma(t) \gamma(t)$, where $\velocity \gamma(t)$ is the score and can be used as a velocity in the affine geometry. This paper demonstrates the feasibility of the Fisher construction for the closed simplex. Specifically, we show that a model can hit the boundary of the simplex while remaining smooth, and that common analytic algorithms have an algebraic, hence computable, form in the affine setting. An example is the gradient flow minimisation algorithm that requires a proper definition of gradient and velocity.

Statistical models are traditionally presented in parametric form and, more rarely, in implicit form. The explicit introduction of a larger structure, the statistical bundle, that includes both probability functions and scores, provides what, in physics, is called a state space. A state space provides an additional way to describe statistical models as solutions to a system of (possibly implicit) differential equations. An example is the notion of a geodesic. The presentation of a model in differential form has direct applications to modelling (e.g., epidemic models such as the SIR model), to approximate computation (e.g., variational inference), and to physics-aware modelling (e.g., via models inspired by Lagrangian mechanics).

\bibliographystyle{plainnat}

\appendix

\section{Appendix}
\label{appendices}

In this appendix we collect technical results, proofs, concepts, and details that aid understanding but distract from the main flow of the paper.

\subsection{Proofs of results in Section~\ref{section-closed-probability-simplex}} \label{appendix_proj} 

\begin{proof} of Proposition~\ref{prop:projective-manifold}. 
The sets $V_x$, $x \in \Omega$ form a covering of $\Projective$ because at least one component has to be non-zero. The image under $\phi_x$ of $V_x \cap V_y$, $x \neq y$, is the set of all vectors in $\reals^{\Omega\setminus\set{x}}$ whose $y$-component is non-zero. Hence it is open. Finally, let us compute $\phi_y \circ \phi_x^{-1} \colon V_x \cap V_y$. We have, by composition,
\begin{equation*}
  \phi_y \circ \phi_x^{-1} \colon (a_z , z \neq x) \mapsto [a_z \colon z \neq x \colon a_x = 1] \mapsto (a_z/a_y , z \notin \set{x,y}, a_x=1/a_y) \ ,
\end{equation*}
which is rational, hence infinitely differentiable where defined.
\end{proof}

\begin{proof} of Theorem~\ref{thm:tangent-bundle} 
Let $\pi$ be the quotient mapping from $f \in \reals^\Omega_*$ to the line $\reals f = [f]$. Let $t \mapsto f(t)$ be a differentiable curve and consider $t \mapsto [f(t)] = \pi(f(t))$. The derivative is defined in a chart. On $V_x$, that is, if $f_x(0) \neq 0$, the representation in $x$-chart on the quotient mapping is
  \begin{equation*}
      t \mapsto \pi_x(f(t)) = (f_y(t)/f_x(t), y \neq x)
    \end{equation*}
    and the derivative is
  \begin{equation} \label{eq:tangent-vector}
    \derivby t \pi_x(f(t)) = \left(\frac{\dot f_y(t) f_x(t) - f_y(t) \dot f_x(t)}{f_x(t)^2}, y \neq x \right) \ .
  \end{equation}

  If $f(t) = \lambda(t) f \in [f]$, then $\dot f(t) = \dot \lambda(t) f$, that is, the velocity belongs to the line. As a consequence, the numerator in Equation \eqref{eq:tangent-vector} cancels:
  \begin{equation*}
      \dot f_y f_x - f_y \dot f_x(t) = \dot \lambda f_y \lambda f_x - \lambda f_y \dot \lambda f_x = 0 \ .
  \end{equation*}
  In conclusion, we can substitute $\hat f$ to $\dot f$ in Equation \eqref{eq:tangent-vector}:
   \begin{equation} \label{eq:tangent-vector-true}
    \derivby t \pi_x(f(t)) = \left(\frac{\hat f_y(t) f_x(t) - f_y(t) \hat f_x(t)}{f_x(t)^2}, y \neq x \right) \ .
  \end{equation}

It can be easily checked that the linear map $\hat f \mapsto  (\hat f_y f_x - f_y \hat f_x(t), y \neq x)$ is 1-to-1. Hence, the tangent vector is identified in the $x$-chart to
\begin{equation} \label{eq:hom}
    [f] \mapsto \left(\hat f \mapsto (\hat f_y f_x - f_y\hat f_x, y \neq x)f_x^{-2}\right)
  \end{equation}

To conclude, we need to show the linearity and smoothness of Equation \eqref{eq:hom}. Let $g = \mu f$, with $\mu = \mu(t)$ a nowhere-zero scalar function, so that $[g(t)] = [f(t)]$ for all $t$. Denote by $\lambda_f, \hat f$ and $\lambda_g, \hat g$ the quantities $\lambda, \hat f$ of Equation \eqref{eq:tangent-vector-true} computed along $f$ and along $g$, respectively. Then $\dot g = \dot\mu f + \mu \dot f$, and
\begin{equation*}
  \lambda_g = \frac{\bracketat g {\dot g} g}{\bracketat g g g} = \frac{\bracketat {\mu f} {\dot\mu f + \mu \dot f} {\mu f}}{\bracketat {\mu f}{\mu f}{\mu f}}  =\frac{\mu\dot\mu \langle f,f \rangle + \langle \dot f, f \rangle}{\langle f,f \rangle} = \frac{\dot\mu}\mu + \lambda_f \ .
\end{equation*}
So
\begin{equation*}
  \hat g = \dot g - \lambda_g g = \big(\dot\mu f + \mu \dot f\big) - \Big(\frac{\dot\mu}\mu + \lambda_f\Big)\mu f = \dot\mu f + \mu \dot f - \dot\mu f - \mu \lambda_f f = \mu\big(\dot f - \lambda_f f\big) = \mu\, \hat f \ .
\end{equation*}
As the $\dot\mu$ terms cancel exactly, the linearity $\widehat{\mu f} = \mu\, \hat f$ holds.
\end{proof}

\subsection{Compositions}\label{appendix_compositional}

Let $\Projective_+$ be the set of compositions on $\Omega$, that is, $f \in \Projective$ and $f_xf_y \geq 0$ for all $x,y \in \Omega$. This property is actually projective, hence it indeed defines a subset of $\Projective$. The image of $V_x \cap \Projective_+$ in the chart $\phi_x$ is the non-negative non-zero quadrant $\phi_x(V_x \cap P_+) = \reals^{\Omega\setminus\set{x}}_+ \setminus \set{0}$.

It is easily seen that the projective geometry of $\Projective_+$ depends on a single compatible chart.

\begin{prop}
  The domain $V_+ = \setof{[f] \in \Projective}{\sum_y f_y \neq 0}$ is open in the projective manifold because $\phi_x(V_x \cap V_+) = \setof{(f_y/f_x, y \neq x)}{\sum_y f_y/f_x > 0}$. The mapping $\phi_+ \colon V_+ \to 1^\perp$,
\begin{equation*}
  V_+ \ni [f_y \colon y \in \Omega] \mapsto \left(\frac{f_y}{\sum_y f_y} - 1, y \in \Omega\right) \ ,
\end{equation*}
is a chart compatible with the initial atlas.
\end{prop}

If a curve is differentiable in the projective differentiable manifold and moreover represents compositional data, then the derivative in Equation \eqref{eq:tangent-vector} is constrained by non-negativity. Namely, it is free if $f_y > 0$, otherwise it is directed along the face of the quadrant. This crucial property has been used in the context of general measure models to define the tangent bundle in \cite{ay|jost|le|schwachhofer:2014}. See also \cite{ay|jost|le|schwachhofer:2017}.

Precisely, in Equation \eqref{eq:tangent-vector} we must have $\dot f_y = 0$ whenever $f_y = 0$. This algebraic condition implies the existence of a factor $\velocity f$ such that $\dot f = \velocity f f$. According to the conventional statistical naming, such a factor is called \emph{the Fisher score} because it coincides with the log derivative for all $y$ such that $f_y \neq 0$. Let us discuss the relation of the Fisher score with the tangent bundle as derived in Theorem \ref{thm:tangent-bundle}.

Under the non-negativity assumption, with $\dot f = \velocity f f$, the derivative of the curve in the $\phi_+$ chart becomes
\begin{equation*}
  \left. d \phi_+(f(t)) \right|_{t=0}= \left(\frac{\dot f_y f_y - f \sum_z \dot f_z}{(\sum_z f_z)^2},y \in \Omega\right) = \left( \left(\velocity f_y - \frac{\sum_z \velocity f_z f_z}{(\sum_z f_z)^2}\right) \frac{f_y}{\sum_z f_z}, y \in \Omega \right) \ .
\end{equation*}

Notice that in the case of a curve in probability functions, $\sum_y f_y = 1$ and $\sum_y \velocity f_y f_y = 0$, so that the derivative in the chart reduces to $\dot f = \velocity f f$.

In conclusion, $\Projective$ is a differentiable manifold with respect to the atlas of homogeneous coordinates. This defines the geometry. The set of compositional data (unnormalised weights) is a subset of $\Projective$ where the Fisher score is well defined up to $f$-equivalence. We will use this set-up with reference to the closed probability simplex, taken as a presentation of $\Projective_+$.

We now further elaborate the comment above. Let us define the derivative by restricting it to $\Projective_+$. Consider the open domain
\begin{equation*}
  U_0 = \setof{f \in \reals^\Omega}{\sum_y f_y = 1 \cdot f \neq 0} 
\end{equation*}
and map it to homogeneous coordinates $\pi \colon \reals^\Omega \setminus \set{0} \to \Projective$. The image is the domain
\begin{equation*}
  V_0 = \pi(U_0) = \setof{[f] \in \Projective}{\sum_y f_y \neq 0} \ .
\end{equation*}

In the chart 
\begin{equation*}
\phi_0 \colon  V_0 \ni [f_y \colon y \in \Omega] \mapsto \left(\frac{f_y}{\sum_y f_y} - 1, y \in \Omega\right) \ ,
\end{equation*}
the equivalence map $\pi \colon \reals^\Omega \setminus \set{0}$ becomes
\begin{equation*}
  \phi_0 \circ \pi \colon U_0 \ni f \mapsto \frac f {1 \cdot f} -1 \in 1^\perp \ . 
\end{equation*}

Let us compute the derivative along a curve $t \mapsto f(t)$. On the curve it holds
\begin{multline*}
    \derivby t \phi_0 \circ \pi(f(t)) = \left(\derivby t \frac {f_x(t)} {\sum_y f_y(t)} , x \in \Omega
    \right) = \\ \left(\frac{\dot f_x(t) \sum_y f_y (t) - f_x(t) \sum_y \dot f_y(t)}{(\sum f_y(t))^2}, x \in \Omega \right)
 = \left(\frac{\dot f_x(t)} {\sum_y f_y (t)} - \frac{f_x(t)}{\sum_y f_y(t)} \frac{\sum_y \dot f_y(t)}{\sum f_y(t))}, x \in \Omega \right) \ ,\end{multline*}
 hence on the trivial tangent bundle, for all $(f,v) \in U_0 \times \reals^\Omega$
 \begin{equation*}
     d \phi_0 \circ \pi (f) \colon v \mapsto \left(\frac{v_x} {\sum_y f_y} - \frac{f_x}{\sum_y f_y} \frac{\sum_y v_y}{\sum f_y}, x \in \Omega \right) \ .
 \end{equation*}

 The kernel depends on the coordinate in the atlas, 
 \begin{equation*}
     \ker d \phi_0 \circ \pi(f) = \setof{v}{v - \frac f {1 \cdot f} (1 \cdot v)} = \setof{v}{v - \phi_0[f] (1 \cdot v)} \ . 
 \end{equation*}
and it is precisely the projective point, $[f] = \reals f$. In fact, the equation implies $u \propto f$ and, conversely, if $v = \lambda f$, then
 \begin{equation*}
     v - \phi_0[f] 1 \cdot v = \lambda f - \frac f {1 \cdot f} \ 1(\lambda f) = 0 \ .
 \end{equation*}

 Any complement of $\reals f$ would produce a unique representative of the derivative. In particular, we can split orthogonally and define as representative of a generic tangent vector
 \begin{equation*}
 [f] \mapsto \hat v = v - \frac{v \cdot f}{f \cdot f} \in T \Projective_0 \ .    
 \end{equation*}
 If the original tangent vector is the ordinary velocity,
 \begin{equation*}
     \widehat {\dot f(t)} = \dot f(t) - \frac {\dot f(t) \cdot f(t)}{f(t) \cdot f(t)} f(t) = \dot f - \frac 12 \derivby t \log \normof{f(t)}^2 \ f(t) \ .
 \end{equation*}

 If $\dot f(t) = \velocity f(t) f(t)$, then
 \begin{equation*}
     \widehat {\dot f(t)} = \left(\velocity f(t) - \frac 12 \derivby t \log \normof{f(t)}^2 \right) f(t) \ .
 \end{equation*}
 
\subsection{Geodesics} \label{sect:parallel}

In the ordinary geometry of the open probability simplex the equations $\ddot \gamma(t)=0$, $\dot \gamma(t) = \dot \gamma(0)$, and $\gamma(t) = \gamma(0) + t \dot \gamma(0)$ are equivalent and define special curves, affinely parametrised, which are usually called geodesics. The same properties hold in general affine spaces, in particular in the statistical bundle as defined in Section~\ref{sec:statistical-bundle}. The main difference is the fact that the tangent spaces at different points are distinct, and we require a weaker notion of constant velocity and a special choice of coordinates. This is solved by introducing mappings between the fibers that provide an identification by "transporting" a score at a point $p$ to a score at a point $q$. We consider the special cases that are most natural in Statistics, the exponential and affine transports. See,~\cite{amari:87} and~\cite{kass|vos:1997}, where the fibers are not centred and the name is the Hilbert bundle. See also the tutorials~\cite{chirco|pistone:2022} and~\cite{pistone:2020-NPCS}.

As in Definition \ref{def:displacement}, for all $p,q \in \Delta(\Omega)$, the \emph{exponential parallel transport} is
\begin{equation*}
    \label{eq:parallel-transport}
    \etransport p q v = v - \expectat q v \ .
\end{equation*}    
The exponential parallel transport trivially satisfies
\begin{gather*}
\etransport p q \colon \expfiberat p \Omega \to \expfiberat q \Omega \ , \\
    \etransport q r \etransport p q = \etransport p r \ , \quad \etransport p p = I \ .
\end{gather*}    

An \emph{exponential geodesic} with initial velocity $\velocity q(0)$ and initial position $q(0)$ has constant velocity when we identify the velocity spaces via the transport,
\begin{equation*}
    \velocity q(t) = \etransport {q(0)} {q(t)} \velocity q(0) \ ,
\end{equation*}
that is,
\begin{equation} \label{eq:e-geodesic}
    \dot q(t) = \left(\etransport {q(0)}{q(t)} \velocity q(0)\right) q(t) = \left(\velocity q(0) - \expectat {q(t)} {\velocity q(0)}\right) q(t) \, .
\end{equation}
The solution of Equation~(\ref{eq:e-geodesic}) with the given initial conditions is the exponential family
\begin{equation} \label{eq:e-geodesic-solution}
    q(t) = \euler^{t \velocity q(0) - \psi(t)} \cdot q(0) \, .
\end{equation}

Notice that $\supp q(t) = \supp  q(0)$, hence
\begin{equation*}
    \psi(t) = \sum_{q(0) > 0} q(0) \log \frac {q(0)}{q(t)} = \KL{q(0)}{q(t)} \, .
\end{equation*}
If we take $t = 1$, then the \emph{exponential displacement} $s_p(q)$ from $p$ to $q$ is defined by
\begin{equation*}
    \label{eq:e-displacement}
    q = \euler^{s_p(q) - \KL{p}{q}} \cdot p \quad \text{and} \quad \expectat p {s_p(q)} = 0 \, .
\end{equation*}
That is,
\begin{equation*}
    \label{eq:exponential-chart}
    s_p(q) = \begin{cases}
        \log \frac q p - \expectat p {\log \frac q p} & \text{on $\supp p$} \\ \text{arbitrary} & \text{otherwise}
    \end{cases}
\end{equation*}
provided $\supp p = \supp q$.

Conversely, in the chart, $s_{q_0}(q(t)) = t \ \velocity q_0$ shows the affine form. For the second derivative, we compute 
\begin{equation*}
\lim_{h \to 0} h^{-1} \left(\etransport {q(t+h)}{q(t)} \velocity q(t+h) - \velocity q(t)\right) \ ,
\end{equation*}
which is zero for the geodesic.

\cite{amari|nagaoka:2000} have introduced the dual of the exponential transport. The covariance is a metric on $\expbundleat{\Omega}$, defined as
\begin{equation}
    \label{eq:duality}
    (q,v,w) \mapsto \scalarat q v w = \covat{q}{v}{w} \ , \quad q \in \Delta(\Omega), v,w \in \expfiberat q {\Omega} \, .
\end{equation}

For all $(p,u) \in \expfiberat p {\Delta(\Omega)}$, and $(q,w) \in \expfiberat q {\Delta(\Omega)}$, assume  $q \mtransport p q u = pu$ and compute,
 \begin{align*}
  \scalarat q {\mtransport p q u} {w} = & \sum_x q(x) \mtransport p q u(x)w(x) = \sum_x p_x u(x) w(x) 
  \\ = & 
  \sum_x p_x u(x) (w(x) - \expectat p w) = \scalarat p u {\etransport q p w}   
 \end{align*}
 where $ \etransport q p w = (\supp p)(w - \expectat p w)  $. 

Next, fix a generic $(p,u)\in \expbundleat \Omega$ and consider the equation for $(q,v)\in \in \expbundleat \Omega$
\begin{equation*}
    pu = qv
\end{equation*}
We have
\begin{enumerate}
    \item If $x \in \supp p^c \cap \supp q^c$, the equation becomes $0=0$ 
    \item \label{item:transport-2} If $x \in \supp p \cap \supp q ^c$, the equation becomes $p_xu(x) = 0$, that is, $u(x)=0$
    \item \label{item:transport-3} If $x \in \supp p^c \cap \supp q$, the equation becomes $q(x)v(x)=0$, that is, $v(x) = 0$
    \item \label{item:transport-4} If $x \in \supp p \cap \supp q$, then $v = \frac pq u$
\end{enumerate}

If we drop the inconsistent Item~\ref{item:transport-2}, that is, we assume $\supp p \subset \supp q$, then Item~\ref{item:transport-3} and Item~\ref{item:transport-4} become
 \begin{align*}
     v(x) &= 0 & x \in \supp q \setminus \supp p \\
     v(x) &= \frac pq u & x \in \supp p
 \end{align*}

 In conclusion, define $\mtransport p q S_p\Delta(\Omega) \to S_q\Delta(\Omega)$ for $\supp p \subset \supp q$ with
$ pu = q\mtransport p q u $ 
to include the boundary of $\Delta$.

\subsubsection{The score as a velocity in the moving frame}

To conclude this section we observe that the score equals the affine velocity in the moving frame provided the support is constant.
Consider a curve $t \mapsto \gamma(t)$ with constant support $\supp{\gamma(t)}=S$. Then, for all $p\in \Delta(\Omega)$ such that $\supp p = S$ 
\begin{equation*}
    \derivby t s_p(\gamma(t)) = \frac{\dot\gamma(t)}{\gamma(t)} - \expectat p {\frac{\dot\gamma(t)}{\gamma(t)}} \quad \text{on $S$, and $0$ otherwise.}
\end{equation*}
In the \emph{moving frame}, $p = \gamma(t)$, so that
\begin{equation*}
\left.\derivby t s_p(\gamma(t))\right|_{p=\gamma(t)} = \frac{\dot \gamma(t)}{\gamma(t)} - \expectat {\gamma(t)} {\frac{\dot \gamma(t)}{\gamma(t)}} = \velocity \gamma(t) \ .
\end{equation*}
That is, first compute the derivative with respect to a generic origin of coordinates, then put the origin at the current point. In the more general case, there is no exponential chart and the score is implicitly defined.

\end{document}